\documentclass[10.9pt,a4paper]{amsart}
\usepackage[english]{babel}

\newtheorem{theorem}{Theorem}[section]
\newtheorem{conj}{Conjecture}
\newtheorem{corollary}[theorem]{Corollary}
\newtheorem{lemma}[theorem]{Lemma}
\newtheorem{proposition}[theorem]{Proposition}

\newtheorem{claim}{Claim}

\theoremstyle{definition}
\newtheorem{definition}[theorem]{Definition}

\newtheorem{example}[theorem]{Example}
\newtheorem{remark}[theorem]{Remark}

\usepackage[square,sort,comma,numbers]{natbib}

\usepackage[colorlinks,
    linkcolor={red!50!black},
    citecolor={blue!50!black},
    urlcolor={blue!80!black}]{hyperref}

\usepackage{url}
\usepackage[all]{xy}
\usepackage{pstricks}
\usepackage{enumerate}
\usepackage{amsfonts,amssymb,amsmath,pinlabel,array,hhline}
\usepackage{slashed}
\usepackage{tabulary}
\usepackage{fancyhdr}
\usepackage{a4wide}
\usepackage{bbm,dsfont}
\usepackage[position=b]{subcaption}
\usepackage{graphicx}
\usepackage{enumitem}
\usepackage{soul}

\newcommand{\Z}{\mathds{Z}}
\newcommand{\Q}{\mathds{Q}}
\newcommand{\R}{\mathds{R}}
\newcommand{\C}{\mathds{C}}

\newcommand{\F}{\mathbb{F}}

\newcommand{\pr}{\operatorname{pr}}
\newcommand{\ord}{\operatorname{ord}}
\newcommand{\id}{\operatorname{id}}

\newcommand{\diag}{\operatorname{diag}}

\newcommand{\LC}{\mathbb{C}[t^{\pm 1}]}
\newcommand{\Bl}{\operatorname{Bl}}

\renewcommand{\setminus}{\smallsetminus}

\newcommand{\bigsharp}{\mathop{\mathchoice
  {\vcenter{\hbox{\LARGE$\#$}}}
  {\vcenter{\hbox{\large$\#$}}}
  {\vcenter{\hbox{\footnotesize$\#$}}}
  {\vcenter{\hbox{\scriptsize$\#$}}}
}\displaylimits}

\DeclareSymbolFont{EulerScript}{U}{eus}{m}{n}
\DeclareSymbolFontAlphabet\mathscr{EulerScript}
\begin{document}

\title{Non-slice linear combinations of iterated torus knots}
\author{Anthony Conway}
\address{Max-Planck-Institut f\"ur Mathematik Vivatsgasse 7, 53111 Bonn, Germany}
\email{anthonyyconway@gmail.com}
\author{Min Hoon Kim}
\address{Department of Mathematics, Chonnam National University, Gwangju 61186, Republic of Korea}
\email{minhoonkim@jnu.ac.kr}
\author{Wojciech Politarczyk}
\address{Institute of Mathematics, University of Warsaw, ul. Banacha 2,
02-097 Warsaw, Poland}
\email{wpolitarczyk@mimuw.edu.pl}
\maketitle

\begin{abstract}
In 1976, Rudolph asked whether algebraic knots are linearly independent in the knot concordance group.
This paper uses twisted Blanchfield pairings to answer this question in the affirmative for new large families of algebraic knots.
\end{abstract}

\section{Introduction}
A knot is \emph{algebraic} if it arises as a link of an isolated singularity of a complex curve. 
Algebraic knots are special cases of iterated torus knots.
In 1976, Rudolph \cite{Rudolph} asked for a description of the subgroup of the knot concordance group generated by algebraic knots.
For ease of reference, we refer to this question as a conjecture.

\begin{conj}[Rudolph's Conjecture {\cite{Rudolph}}]\label{conj:rudolph-conj}
  The set of algebraic knots is linearly independent in the smooth knot concordance group~\(\mathcal{C}\).
\end{conj}

\noindent This question has been of particular interest due to its relevance to the slice-ribbon conjecture:
a result of Miyazaki shows that non-trivial linear combinations of iterated torus knots are not ribbon~\cite[Corollary~8.4]{Miyazaki}.
In particular, if the slice-ribbon conjecture holds, then Rudolph's conjecture holds.
Baker~\cite{Baker} and Abe-Tagami~\cite{AbeTagami} recently noticed that the slice-ribbon conjecture implies a statement stronger than Rudolph's conjecture:

\begin{conj}[Abe-Tagami~\cite{AbeTagami} and Baker~\cite{Baker}]\label{conj:main-conjecture}
  The set of prime fibered strongly quasi-positive knots is linearly independent in the smooth knot concordance group \(\mathcal{C}\).
\end{conj}


This paper exhibits new large families of knots for which Conjectures~\ref{conj:rudolph-conj} and~\ref{conj:main-conjecture} hold.

\subsection{Statement of the results}
Evidence of Rudolph's conjecture was first provided in 1979 by Litherland, who proved that positive torus knots are linearly independent in~$\mathcal{C}$~\cite{Litherland-signature}. 
In~2010, Hedden, Kirk and Livingston showed that for an appropriate choice of positive integers~$\{q_n\}_{n=1}^\infty$, the set~$\{T(2,q_n),T(2,3;2,q_n) \}_{n= 1}^\infty$ is linearly independent in~$\mathcal{C}$, where~$T(p,q)$ and~$T(p,q;r,s)$ denote the~$(p,q)$-torus knot and the~$(r,s)$-cable of~$T(p,q)$, respectively, and~$p$ is coprime to~$qrs$. 
It is known that an \emph{iterated torus knot}~$T(p_1,q_1;\ldots;p_k,q_k)$ is algebraic if and only if~$p_i,q_i>0$ and~$q_{i+1}>q_ip_{i+1}p_i$ for each~$i$.
Our main result, which relies on metabelian twisted Blanchfield pairings~\cite{MillerPowell, BorodzikConwayPolitarczyk}, reads as follows.

\begin{theorem}
  \label{thm:Main}
Fix a prime power~$p$. 
  Let \(\mathcal{S}_{p}\) be the set of iterated torus knots \(T(p,q_{1};p,q_{2};\ldots;p,q_{\ell})\), where the sequences~$(q_{1},q_{2},\ldots,q_{\ell})$ of positive integers that are coprime to~$p$ satisfy
  \begin{enumerate}
  \item $q_\ell$ is a prime;
  \item for \(i=1,\ldots,\ell-1\), the integer \(q_{i} \) is coprime to \(q_{\ell}\) when~$\ell >1$;
  \end{enumerate}
  The set \(\mathcal{S}_{p}\) is linearly independent in the topological knot concordance group~$\mathcal{C}^{\text{top}}$.
\end{theorem}

As an immediate corollary of Theorem~\ref{thm:Main}, we obtain the following.

\begin{corollary}
  For every prime power~$p$, the subset \(\mathcal{S}_{p}^{alg} \subset \mathcal{S}_{p}\) of algebraic knots in~$\mathcal{S}_{p}$ is linearly independent in~$\mathcal{C}^{\text{top}}$ and therefore satisfies Conjecture~\ref{conj:rudolph-conj}.
\end{corollary}

Since positively iterated torus knots are strongly quasi-positive (via~\cite[Theorem 1.2]{HeddenSomeRemarks} and~\cite[Proposition 2.1]{HeddenNotions}), Theorem~\ref{thm:Main} also gives infinite families of knots satisfying Conjecture~\ref{conj:main-conjecture}.

\begin{corollary}
  \label{cor:sqp}
  For every prime power \(p\), the set \(\mathcal{S}_{p}\) satisfies Conjecture~\ref{conj:main-conjecture}, and \(\mathcal{S}_{p} \setminus \mathcal{S}_{p}^{alg}\) is an infinite family of non-algebraic knots satisfying Conjecture~\ref{conj:main-conjecture}.
\end{corollary}

Abe and Tagami also conjecture that the set of L-space knots is linearly independent in~$\mathcal{C}$~\cite[Conjecture 3.4]{AbeTagami}.
For a knot~$K$ with Seifert genus~$g$, the~$(p,q)$-cable~$K_{p,q}$ is an L-space knot if and only if~$K$ is an L-space knot and~$(2g - 1)p \leq q$~\cite{HeddenOnKnotFloer,HomANoteOnCabling}.
Since torus knots are L-space knots, we also obtain the following result.
\begin{corollary}
  For every prime power \(p\), the subset \(\mathcal{S}_{p}^{L} \subset \mathcal{S}_{p}\) of L-space knots in~$\mathcal{S}_{p}$ is linearly independent in~$\mathcal{C}^{\text{top}}$, and this statement also holds for the infinite family~$\mathcal{S}_{p}^{L} \setminus \mathcal{S}_{p}^{\text{alg}}$ of non-algebraic L-space knots.
\end{corollary}

Note however that not all our examples are L-spaces knots: since the cable of an iterated torus knot need not be an L-space knot, Corollary~\ref{cor:sqp} shows that the infinite set~$\mathcal{S}_p \setminus \mathcal{S}_p^L$ contains no L-spaces knots but is nevertheless linearly independent in~$\mathcal{C}^{\text{top}}$.

\subsection{Context and comparison with smooth techniques}

Litherland used the Levine-Tristram signature to show that torus knots are linearly independent in~$\mathcal{C}$~\cite{Litherland-signature}.
This approach is insufficient to answer Rudolph's conjecture since Livingston and Melvin showed in~\cite{LivingstonMelvinAlgebraicKnots} that the following linear combinations of iterated torus knots are algebraically slice:
\begin{equation}
  \label{eq:HKLIntro}J(p,q,q_1,q_2):= T(p,q;p,q_1) \# -T(p,q_1)\# -T(p,q;p,q_2)\# T(p,q_2).
\end{equation}
Classical knot invariants can thus not obstruct~$J(p,q,q_1,q_2)$ from being slice.

Hedden, Kirk, and Livingston managed to leverage the Casson-Gordon invariants to provide further evidence of Rudolph's conjecture~\cite{HeddenKirkLivingston}.
Indeed, they showed that for an appropriate choice of~$\{q_n\}_{n=1}^\infty$, the knots~$\{J(2,3,q_{2n-1},q_{2n})\}_{n= 1}^\infty$ generate an infinite rank subgroup in~$\mathcal{C}$. 
This result is particularly notable since they observe that the~$s$-invariant from Khovanov homology and the~$\tau$-invariant from Heegaard-Floer homology both vanish on~$J(2,3,q_{2n-1},q_{2n})$~\cite[Proposition~8.2]{HeddenKirkLivingston}.
In fact, their argument (combined with Proposition~\ref{prop:algebraic-sliceness}) generalises to show that if \(K\) is a linear combination of algebraically slice knots belonging to \(\mathcal{S}_{p}\), then~$\tau(K) = 0$ and $s(K)=0$.

Next, we observe that the Upsilon invariant~$\Upsilon_K \colon [0,2] \to \R$ from knot Floer homology~\cite{OSS} is also insufficient to prove Theorem~\ref{thm:Main}.
First note that if~$q_1,q_2>p(p-1)(q-1)$, then~$T(p,q;p,q_i)$ is an L-space knot~\cite{HeddenOnKnotFloer}, and thus a result of Tange shows that~$\Upsilon_{T(p,q;p,q_i})(t)=\Upsilon_{T(p,q)}(pt)+\Upsilon_{T(p,q_i)}(t)$ for all~$t\in[0,2]$~\cite[Theorem 3]{Tange}.
The additivity of~$\Upsilon$ then establishes that~$\Upsilon_{J(p,q,q_1,q_2)}(t)=~0$ for all~$t \in [0,2]$ whenever~$q_1,q_2>p(p-1)(q-1)$.

\subsection{Strategy and ingredients of the proof}

The proof of Theorem~\ref{thm:Main} relies on Casson-Gordon theory~\cite{CassonGordon1,CassonGordon2,KirkLivingston}, and more specifically on the metabelian Blanchfield pairings introduced by Miller-Powell~\cite{MillerPowell} and further developed by the first author, the third author, and Maciej Borodzik~\cite{BorodzikConwayPolitarczyk}.
Since these invariants are somewhat technical, the next paragraphs describe some background and ideas that go into the proof of Theorem~\ref{thm:Main}.
For notational simplicity, however, we restrict ourselves to a very particular case: we apply our strategy to the knot~$J(p,q,q_1,q_2)$ described in~\eqref{eq:HKLIntro}.


\subsubsection*{The sliceness obstruction}

Let~$p$ be a prime power, let~$\Sigma_p(J)$ be the~$p$-fold branched cover of the knot~$J:=J(p,q,q_1,q_2)$, let~$\chi$ be a character on~$H_1(\Sigma_p(J))$, and let~$M_J$ be the~$0$-framed surgery of~$J$.
Associated to this data, there is a non-singular sesquilinear and Hermitian \emph{metabelian Blanchfield pairing} 
$$\operatorname{Bl}_{\alpha(p,\chi)}(J) \colon H_1(M_J;\LC^p)\times  H_1(M_J;\LC^p) \to \C(t)/\LC.$$
Here~$H_1(M_J;\LC^p)$ denotes the homology of~$M_J$ twisted by a metabelian representation \linebreak~$\alpha(p,\chi) \colon \pi_1(M_J) \to GL_p(\LC)$ whose definition will be recalled in Section~\ref{sec:TwistedPolynomial}.
The precise definition of~$\operatorname{Bl}_{\alpha(p,\chi)}(J)$ is irrelevant in this paper: only its properties are required.
Informally, however, the pairing~$\operatorname{Bl}_{\alpha(p,\chi)}(J)~$ contains both the information from twisted polynomial invariants and twisted signature invariants.
We now describe how~$\operatorname{Bl}_{\alpha(p,\chi)}(J)~$ provides a sliceness obstruction.

Let~$\lambda_p(J)$ denote the~$\Q/\Z$-valued linking form on~$H_1(\Sigma_p(J))$.
Miller and Powell show that if for every~$\Z_p$-invariant metaboliser~$G$ of~$\lambda_p(J)$, there exists a prime power order character~$\chi$ that vanishes on~$G$ and such that~$\Bl_{\alpha(p,\chi)}(J)$ is not metabolic, then~$J$ is not slice~\cite[Theorem~6.10]{MillerPowell}.
In order to make this obstruction more concrete, we now recall some terminology on linking forms and their metabolizers.

\subsubsection*{The Witt group of linking forms}

We focus on linking forms over~$\LC$, referring to Section~\ref{sec:Metabolisers} for a discussion over more general rings.
A \emph{linking form} over~$\LC$ is a sesquilinear Hermitian pairing~$V \times V \to \C(t)/\LC$, where~$V$ is a torsion~$\LC$-module.
A linking form~$(V,\lambda)$ is \emph{metabolic} if there is a submodule~$L \subset V$ such that~$L=L^\perp$; such an~$L$ is called a \emph{metaboliser}.
The \emph{Witt group of linking forms}, denoted~$W(\C(t),\LC)$, consists of the monoid of non-singular linking forms modulo the submonoid of metabolic linking forms.
We write~$\lambda_1 \sim \lambda_2$ if two linking forms agree in~$W(\C(t),\LC)$.
The Miller-Powell obstruction to sliceness, therefore, consists of deciding whether a certain twisted Blanchfield pairing~$\Bl_{\alpha(p,\chi)}(J)$ is zero in the group~$W(\C(t),\LC)$.
As we will now describe, one of our main ideas is to transfer a problem of linear independence in~$\mathcal{C}^{\text{top}}$ (namely Rudolph's conjecture) into a problem of linear independence in~$W(\C(t),\LC)$.

\subsubsection*{From linear independence in~$\mathcal{C}^{\text{top}}$ to linear independence in~$W(\C(t),\LC)$}

Since the \linebreak knot~$J= T(p,q;p,q_1) \# -T(p,q_1)\# -T(p,q;p,q_2)\# T(p,q_2)$ is a connected sum of four knots, both~$H_1(\Sigma_p(J))$ and~$\lambda_p(J)$ can be decomposed into four direct summands:
$$\lambda_p(J)=\lambda_p(T(p,q_1))\oplus -\lambda_p(T(p,q_1))\oplus \lambda_p(T(p,q_2))\oplus -\lambda_p(T(p,q_2)).$$
In particular, any character on~$H_1(\Sigma_p(J))$ can be written as~$\chi=\chi_1\oplus\chi_2\oplus\chi_3\oplus\chi_4$. 
For each given~$\Z_p$-invariant metaboliser~$M$ of~$\lambda_p(J)$, the ``sliceness-obstructing character" that we will produce will be of the form~$\chi=\chi_1 \oplus \chi_2 \oplus \theta \oplus \theta$ where~$\theta$ denotes the trivial character.
Using the definition of~$J$, together with the direct sum decomposition of \cite[Corollary~8.21]{BorodzikConwayPolitarczyk}, the Witt class of the metabelian Blanchfield pairing of~$J$ is given by
\begin{align}\label{equation:Bl(J)}
  \operatorname{Bl}_{\alpha(p,\chi)}(J) \sim \operatorname{Bl}_{\alpha(p,\chi_1)}(T(p,q;p,q_1))
  &\oplus -\operatorname{Bl}_{\alpha(p,\chi_2)}(T(p,q_1))  \\
  &\oplus - \operatorname{Bl}_{\alpha(p,\theta)}(T(p,q;p,q_2))\oplus \operatorname{Bl}_{\alpha(p,\theta)}(T(p,q_2)). \nonumber
\end{align}
This expression can be further decomposed by applying the satellite formula for the metabelian Blanchfield forms given in \cite[Theorem~8.19]{BorodzikConwayPolitarczyk}.
Regardless of the final expression, the problem has been converted into a question of linear independence in~$W(\C(t),\LC)$.
In Proposition~\ref{prop:Splitting}, we describe a criterion for linear independence in terms of roots of the orders of the underlying modules (recall that the order of a module is a Laurent polynomial in~$\LC$; it is defined up to multiplication by units of~$\LC$).
Here is a simplified version of this statement.
\begin{proposition}
  \label{prop:SplittingIntro}
  If~$(V_1,\lambda_1)$ and~$(V_2,\lambda_2)$ are two non-metabolic linking forms over~$\LC$ such that~$\operatorname{Ord}(V_1)$ and~$\operatorname{Ord}(V_2)$ have distinct roots, then the Witt classes~$[V_1,\lambda_1]$ and~$[V_2,\lambda_2]$ are linearly independent in \(W(\C(t),\C[t^{\pm 1}])\).
\end{proposition}

\subsubsection{Computation of twisted Alexander polynomials}
In order to apply Proposition~\ref{prop:SplittingIntro}, we must, therefore, understand the roots of the metabelian twisted Alexander polynomials of~$T(p,q)$ associated to characters on~$H_1(\Sigma_p(T(p,q)))$.
This is carried out in Section~\ref{sec:TwistedPolynomial} and relies on our explicit understanding of the~$p$-fold cover~$E_p(T(p,q)) \to E(T(p,q))$ from Section~\ref{sec:branch-covers-torus}.
Since this computation of twisted polynomials might be of independent interest, we summarize it as follows.

\begin{proposition}[Lemma~\ref{lem:Characaters} and Corollary~\ref{cor:TwistedAlexanderPolynomial}]
  \label{prop:TwistedPolyIntro}
  Let~$p,q>0$ be two coprime integers, and set~$\xi_p=e^{2\pi i/p}$.
  The abelian group of characters on \(H_1(\Sigma_{p}(T(p,q))) \cong \Z_q^{p-1}\) is isomorphic~to
  \[\{ \mathbf{a}:=(a_{1},\ldots,a_{p}) \in \Z_{q}^{p} \ | \ a_1+\cdots+a_p=0\}.\]
  We write~$\chi_{\mathbf{a}}$ for the character associated to~$\mathbf{a}$.
  The metabelian twisted Alexander polynomial of the~$0$-framed surgery~$M_{T(p,q)}$ associated to the character~$\chi_{\mathbf{a}}\colon H_1(\Sigma_{p}(T(p,q))) \to \Z_q$ is given by
  ~$$\Delta_{1}^{\alpha(p,\chi_{\mathbf{a}})}(M_{T(p,q)})= \frac{(-1)^{p-1}(1-t^{q})^{p-1}}{(t\xi_q^{a_{1}}-1)(t\xi_q^{a_{2}}-1) \cdots (t\xi_q^{a_{p}}-1)(t-1)}.~$$  
\end{proposition}

\subsubsection*{Main steps of the proof}

We now return to the knot $J=J(p,q,q_1,q_2)$ from~\eqref{eq:HKLIntro}.
Obstructing $J$ from being slice has three main steps.
In fact, the proof of Theorem~\ref{thm:Main} in Section~\ref{sec:MainTheorem} follows more complicated versions of these same steps.
\begin{enumerate}
\item Firstly, we use the previously described ingredients to study the implications of $\Bl_{\alpha(p,\chi)}(J)$ being metabolic on the characters $\chi_1$ and $\chi_2$; here $\chi=\chi_1 \oplus \chi_2 \oplus \theta \oplus \theta$ with $\theta$ the trivial character. This is the content of Subsection~\ref{subsub:Blanchfield}.
\item Secondly, we show that for every metaboliser $L$ of~$\lambda_p(T_{p,q_1}) \oplus -\lambda_p(T_{p,q_1})$, it is possible to build characters $\chi_1$ and $\chi_2$ that violate these conditions and such that $\chi_1 \oplus \chi_2$ vanishes on $L$. This is the content of Subsection~\ref{subsub:BuildingCharacters}.
\item Finally, we combine these two steps to obstruct the sliceness of $J$: for every metaboliser~$G$ of~$\lambda_p(J)$, we are able to build a character $\chi=\chi_1 \oplus \chi_2 \oplus \theta \oplus \theta$ that vanishes on~$G$ and such that $\Bl_{\alpha(p,\chi)}(K)$ is not metabolic.
  This is the content of Subsection~\ref{subsub:Conclusion}.
\end{enumerate}

\begin{remark}
  When~$p=2$, Hedden, Kirk and Livingston also use an obstruction based on the Casson-Gordon set-up to show that for an appropriate choice of positive integers~$\{q_n\}_{n=1}^\infty$, the set~$\{T(2,q_n),T(2,3;2,q_n) \}_{n= 1}^\infty$ is linearly independent in~$\mathcal{C}^{\text{top}}$~\cite{HeddenKirkLivingston}.
  Our work differs from theirs in two main points:
  \begin{itemize}
  \item While~\cite{HeddenKirkLivingston} uses a blend of discriminants and signatures to prove its linear independence result,  we use metabelian Blanchfield pairings.
    In a nutshell, the Blanchfield pairing encapsulates both the discriminant and (most of) the signature invariants allowing us to both streamline and generalize several of the argument from~\cite{HeddenKirkLivingston}.
  \item The result of~\cite{HeddenKirkLivingston} is proved without having to study invariant metabolizers; see also~\cite[Section 9]{BorodzikConwayPolitarczyk}.
    This is a feature of iterated torus knots~$T(p,Q)$ with~$p=2$ and fails when~$p>2$.
  \end{itemize}

 Passing from our outline to obstruct the sliceness of~\(J(p,q,q_{1},q_{2})\) to the proof of Theorem~\ref{thm:Main} requires
  additional steps.
As often in Casson-Gordon theory, the main technical difficulty to overcome concerns the metabolizers of the linking form of the knot in question.
  Regarding these metabolizers, our strategy can be summarized as follows:
  \begin{enumerate}
  \item Given a metabolizer, we isolate certain technical conditions which guarantee that a character violates the sliceness obstruction.
    This is the content of Lemma~\ref{lemma:constructing-characters}.
  \item We distinguish a certain family of metabolizers called
    \emph{graph metabolizers}, see~Section~\ref{sub:Graph}.
  \item The construction of the required character, for any fixed
    non-graph metabolizer is not overly challenging, see Cases~1 and~2 in the proof of
    Lemma~\ref{lemma:constructing-characters}.
  \item Dealing with graph metabolizers requires more work. 
  In Case 3, we show that either there exists a character satisfying the conditions
    from Lemma~\ref{lemma:constructing-characters}, or the knot in
    question contains a slice summand~\(K \# -K\), for some knot \(K\).
    Consequently, once we cancel all the summands of the form \(K \#
    -K\), we are able to construct the desired obstructing character for any graph metabolizer, and finish the proof.
  \end{enumerate}
\end{remark}

\subsection{Assumptions and outlook}
We conclude this introduction by commenting on the various technical assumptions that appear in Theorem~\ref{thm:Main}.
\begin{enumerate}
\item The assumption that the integers~$q_i$ are coprime to~$q_\ell$ is used in Proposition~\ref{prop:NotMetabolic} to ensure that certain Witt classes are linearly independent in~$W(\C(t),\C[t^{\pm 1}])$.
  This hypothesis has its roots in the notion of \emph{$p$-independence} introduced in~\cite[Definition 6.2]{HeddenKirkLivingston}.
\item We assume that~$p$ is a prime power in order to use Casson-Gordon theory~\cite{CassonGordon1,CassonGordon2}.
\item We required that the~$q_i$ be positive mostly because of our interest in Rudolph's conjecture: algebraic knots are iterated torus knots with \emph{positive} cabling parameters.
\item We use that the~$q_{i,\ell_i}$ are prime in order to obtain the decomposition in~\eqref{eq:DecompoBranchedCover} and to ensure that~$\F_{q_{i,\ell_i}}$ is a field. 
\end{enumerate}

Summarising, our assumptions are made for technical reasons: we have so far not encountered linear combinations of (algebraically slice) iterated torus knots whose sliceness is not obstructed by some Casson-Gordon invariants. 
Furthermore, this paper does not fully use the techniques developed in~\cite{BorodzikConwayPolitarczyk} to compute the Casson-Gordon Witt class.
Therefore, it would be interesting to study how far these methods can be pushed to investigate Rudolph's conjecture.

\subsection{Organisation}
This paper is organized as follows.
In Section~\ref{sec:branch-covers-torus}, we collect several results on the algebraic topology of the exterior of the torus knot~$T(p,q)$.
In Section~\ref{sec:TwistedPolynomial}, we use these results to compute Alexander polynomials of~$T(p,q)$ twisted by metabelian representations.
In Section~\ref{sec:Metabolisers}, we review some facts about linking forms.
Finally in Section~\ref{sec:MainTheorem}, we prove Theorem~\ref{thm:Main}.

\subsection*{Acknowledgments}
AC thanks the MPIM for its financial support and hospitality. MHK was partly supported by the POSCO TJ Park Science Fellowship and by NRF grant 2019R1A3B2067839.
WP was supported by the National Science Center grant 2016/22/E/ST1/00040.
We wish to thank the CIRM in Luminy for providing excellent conditions where the bulk of this work was carried out.

\subsection*{Conventions}
Manifolds are assumed to be compact and oriented.
Throughout the paper, the~$p$-fold branched cover of a knot is denoted~$\Sigma_p(K)$, and~$\lambda_p(K)$ denotes the linking form on~$H_1(\Sigma_p(K))$.

\section{Branched covers of torus knots}
\label{sec:branch-covers-torus}

The aim of this section is to describe the~$\Z[\Z_p]$-module structure of~$H_1(\Sigma_p(T(p,q)))$ induced from the~$\Z_p$-covering action on~$\Sigma_p(T(p,q))$ when~$q$ is a prime.
Let~$E(T(p,q))$ be the complement of the torus knot~$T(p,q)$, and let~$E_p(T(p,q))$ be its~$p$-fold cyclic cover.
In Subsection~\ref{sub:KnotGroup}, to set up some notation, we recall the decomposition of~$E(T(p,q))$ coming from the standard genus 1 Heegaard splitting of~$S^3$, as described in \cite[Example 1.24]{Hatcher}. 
In Subsection~\ref{sub:BranchedCover}, this decomposition of~$E(T(p,q))$ is used to decompose~$E_p(T(p,q))$: after that,~$H_1(\Sigma_p(T(p,q)))$ can be computed via a Mayer-Vietoris sequence argument since~$\Sigma_p(T(p,q))$ is a union of~$E_{p}(T(p,q))$ with a solid torus glued along the torus boundary.

\subsection{The homotopy type of~$E(T(p,q))$}
\label{sub:KnotGroup} 

The goal of this subsection is to describe the homotopy type of~$E(T(p,q))$, as well as describe explicit generators for~$\pi_1(E(T(p,q)))$.
To achieve this, we follow closely~\cite[Example 1.24]{Hatcher}. 
\medbreak
Consider the standard decomposition~$S^3=S^1\times D^2\cup D^2\times S^1$ and denote~$S^1\times D^2$ and~$D^2\times S^1$ by~$H_1$ and~$H_2$ respectively;~$H_1 \subset \R^3$ being the bounded solid torus. 
We parametrize the~$(p,q)$-torus knot~$T(p,q)$ on the torus~$H_1\cap H_2$ as follows: 
\begin{equation}\label{equation:torus-knot}T(p,q)=\{(e^{2\pi i pt},e^{2\pi iqt})\mid t\in [0,1]\}\subset S^1\times S^1 = H_1\cap H_2.
\end{equation}
Using this description of~$T(p,q)$, for each~$x\in S^1$, we see that~$T(p,q)$ intersects~$\{x\}\times D^2\subset H_1$ in~$p$ equi-distributed points of~$\{x\}\times \partial D^2$; see Figure~\ref{figure:equi-distributedpoints} for~$p=3$.

\begin{figure}[!htb]
  \centering
  \includegraphics{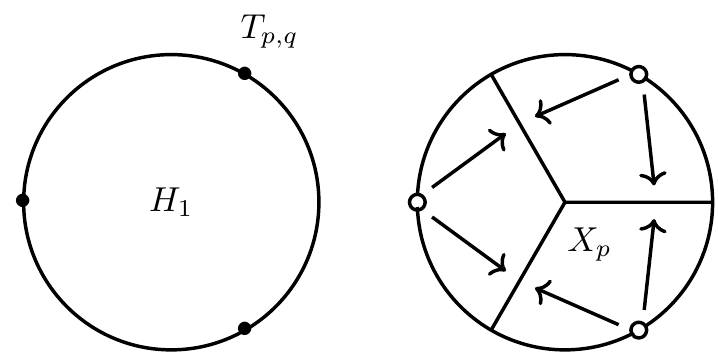}
  \caption{On the left hand side: the intersection~$T(p,q)\cap (\{x\}\times D^2)$. On the right hand side: the complement~$H_1\smallsetminus T(p,q)$ deformation retracts onto a 2-complex~$X_p$.}
  \label{figure:equi-distributedpoints}
\end{figure}

As depicted in the right hand side of Figure~\ref{figure:equi-distributedpoints}, the complement~$H_1\smallsetminus T(p,q)$ deformation retracts onto a 2-complex~$X_p \subset H_1$ which is the mapping cylinder of the degree~$p$ map~$f_p\colon S^1\to c_1$, where~$c_1$ is the core circle of~$H_1$. 
The same argument shows that~$H_2\smallsetminus T(p,q)$ deformation retracts onto the mapping cylinder~$X_q$ of the degree~$q$ map~$f_q\colon S^1\to c_2$, where~$c_2$ is the core circle of~$H_2$. By perturbing~$X_p$ near~$H_1\cap H_2$, we can arrange that~$X_p$ and~$X_q$ match up. 
Next, let~$X_{p,q}$ be the union of~$X_p$ and~$X_q$. Note that~$X_{p,q}$ is homeomorphic to the double mapping cylinder of the maps~$f_p$ and~$f_q$, defined by 
\[X_{p,q}:= S^1\times[0,1]\cup c_1 \cup c_2/{\sim}\]
where~$(z,0)\sim f_p(z)$ and~$(z,1)\sim f_q(z)$ for all~$z\in S^1$ (see Figure~\ref{figure:doublemappingcylinder}). By van Kampen's theorem, 
\[\pi_1(X_{p,q})\cong \langle c_1,c_2\mid c_1^p=c_2^q\rangle.\]
Summarizing, we have the following proposition which is implicit in \cite[Example 1.24]{Hatcher}:

\begin{proposition}[{\cite[Example 1.24]{Hatcher}}]\label{prop:complement-torus-knot}
  There is a deformation retraction~$E(T(p,q))\to X_{p,q}$ sending~$H_1\smallsetminus T(p,q)$ and~$H_2\smallsetminus T(p,q)$ to~$X_p$ and~$X_q$.  In particular,~$\pi_1(E(T(p,q)))\cong \langle c_1,c_2\mid c_1^p=c_2^q\rangle$ where~$c_i$ is the core circle of~$H_i$ for~$i=1,2$. 
\end{proposition}

\subsection{The computation of~$H_1(\Sigma_p(T(p,q)))$ as a~$\Z[\Z_p]$-module}
\label{sub:BranchedCover}

In this subsection, we describe the~$\Z[\Z_p]$-module structure of~$H_1(\Sigma_p(T(p,q)))$. 
To do so, we first study the~$p$-fold cyclic covering map~$\pi\colon E_p(T(p,q))\to E(T(p,q))$, then we compute~$\pi_1(E_p(T(p,q)))$, and finally we describe~$H_1(\Sigma_p(T(p,q)))$.
\medbreak
We first use Subsection~\ref{sub:KnotGroup} to describe a deformation retract of~$E_p(T(p,q))$.
Using~\eqref{equation:torus-knot}, we see that the torus knot~$T(p,q)$ links respectively~$q$ and~$p$ times the core circles~$c_1$ and~$c_2$.
Consequently,~$c_1$ and~$c_2$ are homologous to~$q\mu$ and~$p\mu$ in~$H_1(E(T(p,q)))$, where~$\mu=c_1^kc_2^l$ is a meridian of~$T(p,q)$ and~$pk+ql=1$.
Use~$(X_{p,q})_p$ to denote the pre-image~$\pi^{-1}(X_{p,q})$, and observe that by Proposition~\ref{prop:complement-torus-knot},~$E_p(T(p,q))$ deformation retracts onto~$(X_{p,q})_p$.

\begin{figure}[!htb]
  \centering
  \includegraphics{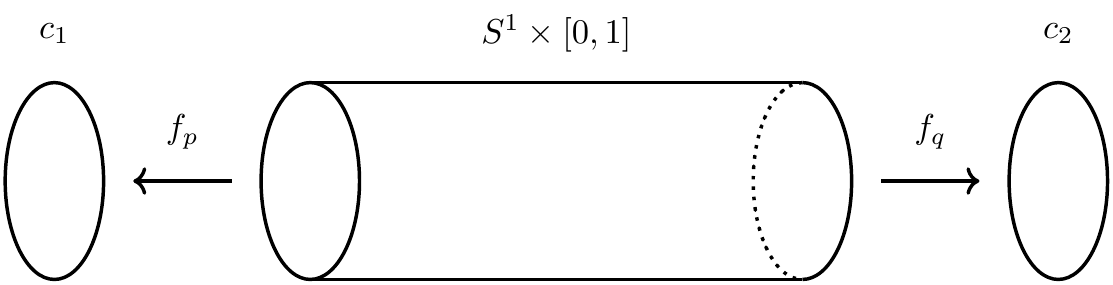}
  \caption{The double mapping cylinder~$X_{p,q}$ obtained by gluing~$S^1\times [0,1]$ with the circles~$c_1$ and~$c_2$ by the degree~$p$ and~$q$ maps~$f_p$ and~$f_q$.}
  \label{figure:doublemappingcylinder}
\end{figure}

We describe~$\pi_1(E_p(T(p,q)))$ by studying the homotopy type of~$(X_{p,q})_p$.
The (restricted) covering map~$\pi\colon (X_{p,q})_p\to X_{p,q}$ corresponds to the homomorphism~$\pi_1(X_{p,q})\to \Z_p$ sending~$c_1$ to~$q\in \Z_p$ and~$c_2$ to~$0\in \Z_p$. 
We use~$\pi_*\colon \pi_1( (X_{p,q})_p)\to \pi_1(X_{p,q})$ to denote the induced map.
Let~$a$ be the pre-image~$\pi^{-1}(c_1)$ and let~$b_0,\ldots,b_{p-1}$ be the components of the pre-image~$\pi^{-1}(c_2)$;
we choose the indices of the~$b_i$'s so that 
\begin{equation}
  \label{eq:bi}
  \pi_*(b_i)=\mu^i c_2\mu^{-i} \quad \text{ for } i=0,\ldots, p-1.
\end{equation}
Since~$\pi$ is a covering map, the induced map~$\pi_*\colon \pi_1( (X_{p,q})_p)\to \pi_1(X_{p,q})$ is injective.
For this reason, we shall often identify~$b_i$ with~$\mu^i c_2\mu^{-i}$.
Since~$X_{p,q}$ is a double mapping cylinder, so is~$(X_{p,q})_p$. 
\begin{figure}[htb!]
  \centering
  \includegraphics{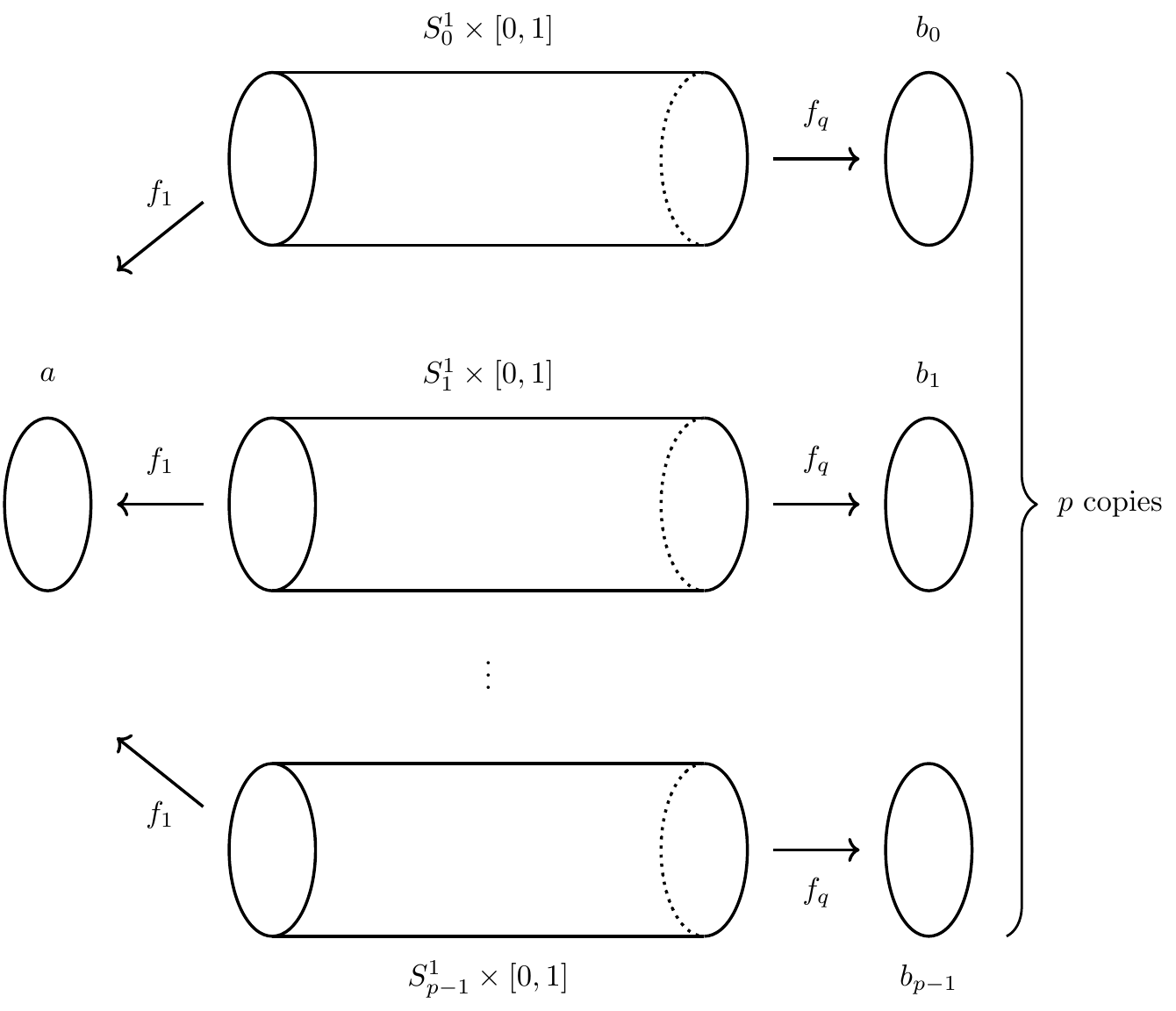}
  \caption{The~$p$-fold cyclic cover~$(X_{p,q})_p$ of~$X_{p,q}$ is also a double mapping cylinder, where~$f_1$ and~$f_q$ denote the degree~$1$ and the degree~$q$ maps, respectively.}
  \label{figure:(X_{p,q})_p}
\end{figure}
More precisely as illustrated in Figure~\ref{figure:(X_{p,q})_p}, we have
\[(X_{p,q})_p=\bigcup_{i=0}^{p-1}S^1_i\times [0,1] \cup a\cup b_0\cup\cdots\cup b_{p-1}/{\sim},\]
where each~$S^1_i\times \{0\}$ is identified with the circle~$a$ by the identity map, and~$S^1_i\times \{1\}$ is identified with the circle~$b_i$ by the degree~$q$ map. By van Kampen's theorem, we deduce that
\[\pi_1((X_{p,q})_p)\cong \langle b_{0},b_{1},\ldots,b_{p-1} \mid b_{i}^{q} = b_{j}^{q} \text{ for }  0 \leq i \neq j \leq p-1 \rangle.\]
Since~$E_p(T(p,q))$ deformation retracts onto~$(X_{p,q})_p$, we obtain the following proposition.

\begin{proposition}
  \label{prop:FundamentalGroupCyclicCover} 
  Let~$\pi\colon E_p(T(p,q))\to E(T(p,q))$ be the~$p$-fold cyclic covering and let~$b_0,b_1,\ldots,b_{p-1}$ be the homotopy classes of the components of~$\pi^{-1}(c_2)$ so that~$\pi_*(b_i)=\mu^i c_2\mu^{-i}$. Then 
  \[\pi_{1}(E_p(T(p,q))) = \langle b_{0},b_{1},\ldots,b_{p-1} \mid b_{i}^{q} = b_{j}^{q} \text{ for }  0 \leq i \neq j \leq p-1 \rangle.\]
\end{proposition}

Next, we use this description of~$\pi_1(E_p(T(p,q)))$ to obtain generators of the finite abelian group~$H_1(\Sigma_p(T(p,q)))=TH_1(E_p(T(p,q)))$.
First, note that Proposition~\ref{prop:FundamentalGroupCyclicCover} shows \linebreak that~\(H_{1}(E_p(T(p,q))) \cong \Z \oplus \Z_{q}^{p-1}\) has generators \(b_0,b_1,\ldots b_{p-1}\) and relations~$q b_{i} = q b_{j}$ for each~$i,j$.
In the remainder of this section, we describe a set of generators that will be more convenient for the twisted Alexander polynomial computations of Section~\ref{sec:TwistedPolynomial}.

\begin{remark}
  \label{rem:LiftMeridians}
  While the meridian~$\mu$ of~$T(p,q)$ does not lift to~$E_p(T(p,q))$, a loop representing~$\mu^p$ does.
  Since the projection induced map~$\pi_* \colon \pi_1(E_p(T(p,q))) \to \pi_1(E(T(p,q)))$ is injective, we slightly abuse notations and also write~$\mu^p$ for the homotopy class of this lift in~$\pi_1(E_p(T(p,q)))$.
\end{remark}

In what follows, we make no notational distinction between elements in~$\pi_1(E_p(T(p,q)))$ and elements in~$H_1(E_p(T(p,q)))$, despite switching from multiplicative to additive notations.
In some rare instances, we will also use the multiplicative notation in homology. 
Keeping this in mind, for~$i=0,\ldots, p-1~$, we consider~$\mu^{-p}b_i$ in~$\pi_1(E_p(T(p,q)))$ and~$x_i:=b_i-\mu^p$ in~$H_1(E_p(T(p,q)))$.
The next proposition describes the homology group~$H_{1}(\Sigma_{p}(T(p,q)))$ as a~$\Z[\Z_p]$-module.

\begin{proposition}\label{prop:homology-group-cover}
  The abelian group \(H_{1}(\Sigma_{p}(T(p,q))) \cong \Z_{q}^{p-1}\)  is generated by the~$x_{i} =b_i-\mu^p$,
  and these elements satisfy the following relations:
  \begin{enumerate}
  \item \(x_{0}+x_{1}+\cdots+x_{p-1}=0\),
  \item \(x_{i} = t^{i}x_{0}\),
    where \(t\) denotes the covering transformation of \(\Sigma_{p}(T(p,q))\).
  \end{enumerate}
  In particular, there exists an isomorphism of \(\Z[t^{\pm1}]\)-modules
  \[H_{1}(\Sigma_{p}(T(p,q))) \cong \Z_{q}[t]/(1+t+t^{2}+\cdots+t^{p-1}).\]
\end{proposition}

\begin{proof}
  The proof has four steps.
  Firstly, we establish a criterion for an element in~$H_1(\Sigma_p(T(p,q)))$ to be torsion; secondly, we prove that the~$x_i$ are torsion; thirdly, we show that that~$x_i$ generate~$TH_1\Sigma_p(T(p,q))$ as an abelian group; fourthly and finally we prove that the~$x_i$ satisfy the two identities stated in the lemma.

  We assert that an element \(x =\sum_{i=0}^{p-1} a_ib_i\) in~$H_{1}(E_p(T(p,q)))$ is torsion if and only if \(\sum_{i=0}^{p-1} a_i=~0\).
  The map~$\pi_* \colon H_1(E_p(T(p,q))) \to H_1(E(T(p,q)))$ maps~$ TH_{1}(E_p(T(p,q)))$ to zero and maps the infinite cyclic summand isomorphically onto~$p\Z \cong \Z \langle c_2 \rangle$. \footnote{For any knot~$K$ and prime power~$n$, one has the decomposition~$H_1(E_n(K))=TH_1(E_n(K)) \oplus \Z$, where the~$\Z$ summand is generated by a lift of the~$n$-fold power of the meridian.}
  In particular, a class~$x \in H_1(E_p(T(p,q)))$ is torsion if and only if~$\pi_*(x)=0$.
  On the other hand, using Proposition~\ref{prop:FundamentalGroupCyclicCover}, we deduce that~$\pi$ induces the following map on homology, concluding the proof of the assertion:
  \begin{align*}
    \pi_* \colon H_{1}(E_p(T(p,q))) &\to p\Z \subset \Z=H_1(E(T(p,q)))\\
    \sum_{i=0}^{p-q} a_ib_i&\mapsto \sum_{i=0}^{p-1} a_i.
  \end{align*}
  We move on to the second step: we prove that the homology classes~$x_0,\ldots,x_{p-1}$ are torsion.
  Using the criterion, we must show that~$\pi_*(x_i)=0$ for each~$i$.
  Since~$\pi_*(b_i)=1$, this reduces to showing that~$\pi_*(\mu^p)=1$.
  We start by computing the abelianization of~$\mu^p$.
  Since~$\mu=c_1^kc_2^l$, we notice that in \(\pi_1(E_p(T(p,q)))\), the following equality holds:
  \begin{equation}
    \label{eq:Mup}
    \mu^{p} = (c_{1}^{k}c_{2}^lc_{1}^{-k}) \cdot (c_{1}^{2k} c_{2}^{l} c_{1}^{-2k})  \cdots  (c_{1}^{(p-1)k} c_{2}^{l} c_{1}^{-(p-1)k}) c_{1}^{pk} c_{2}^{l}.
  \end{equation}
  In order to compute the abelianisation of this expression, we claim that for any~\(0 \leq s \leq p-1\), and any~$k$, the equation
  ~$\mu^{s} c_{2} \mu^{-s} = c_{1}^{ks} c_{2} c_{1}^{-ks}$
  holds in~\(H_{1}(E_p(T(p,q))) = \pi_1(E_p(T(p,q)))^{ab}\).
  This claim is a consequence of following direct computation in \(\pi_1(E_p(T(p,q)))\):
  \[\mu^{s} c_{2} \mu^{-s} = \left(\prod_{i=1}^{s-1} c_{1}^{ki} c_{2}^{l} c_{1}^{-ki}\right) \cdot \left( c_{1}^{ks} c_{2} c_{1}^{-ks} \right) \cdot \left(\prod_{i=1}^{s-1} c_{1}^{ki} c_{2}^{-l} c_{1}^{-ki}\right).\]
  Using consecutively~\eqref{eq:Mup}, the equation~$\mu^{s} c_{2} \mu^{-s} = c_{1}^{ks} c_{2} c_{1}^{-ks}$ that we just established, and the identification~$b_i=\mu^ic_2\mu^{-i}$ from~\eqref{eq:bi} (as well as the presentation in Proposition~\ref{prop:complement-torus-knot} and~\(qk+pl=1\)), we obtain the following sequence of equalities in~$H_1(E_p(T(p,q)))$: 
  \begin{align}
    \label{eq:mup}
    \mu^{p} 
    &=(c_{1}^{k}c_{2}^lc_{1}^{-k}) \cdot (c_{1}^{2k} c_{2}^{l} c_{1}^{-2k})  \cdots  (c_{1}^{(p-1)k} c_{2}^{l} c_{1}^{-(p-1)k}) c_{1}^{pk} c_{2}^{l} \nonumber \\
    &=(\mu c_{2}^l \mu^{-1})(\mu^2 c_{2}^l \mu^{-2})\cdots  (\mu^{(p-1)} c_{2}^{l} \mu^{-(p-1)})c_{1}^{pk} c_{2}^{l} \nonumber \\
    &= l (b_{0}+b_{1}+\cdots+b_{p-1}) + qkb_{0}. 
  \end{align}
  As~$\pi_*(b_i)=1$ for each~$i$, this implies that \(\pi_*(\mu^{p})=1\).
  It follows that \(\pi_*(x_{i}) = \pi_*(b_{i})-\pi_*(\mu^{p})=~0\), and therefore each of the~$x_i$ is torsion.
  This concludes the second step of the proof.

  Thirdly, we show that every element of~$TH_1(E_p(T(p,q)))$ can be written as a linear combination of the~$x_i$ for~$i=0,1,\ldots,p-1$: given \(x = \sum_{i=0}^{p-1}a_ib_i\), adding and substracting~$\mu^p$, using~$\sum_{i=0}^{p-1}a_i=0$ (which holds thanks to the first step) and the definition of~$x_i$, we obtain
  \begin{align*}
    x&=\sum_{i=0}^{p-1}a_ib_i
       =\sum_{i=0}^{p-1}a_i\mu^{p} +\sum_{i=0}^{p-1}a_i(b_i-\mu^{p} )
       =\sum_{i=0}^{p-1}a_ix_i.
  \end{align*}
  Fourthly and finally, we establish the relations~\(x_{0}+x_{1}+\cdots+x_{p-1}=0\) and~$x_i=t^ix_0$.
  The latter relation is clear (since~$b_{i} = t^i b_{0}$ and \(t  \mu^{p} = \mu^{p}\)) and so we focus on the former.
  Using consecutively~\eqref{eq:mup}, the relation \(qb_{i}=qb_{j}\), and the fact that~$pl+qk=1$, we notice that the following equation holds in~\(H_{1}(E_p(T(p,q)))\): 
  \begin{align*}
    p \mu^p &= pl(b_{0}+b_{1}+\cdots+b_{p-1}) + pqkb_{0}  \\
            &= pl(b_{0}+b_{1}+\cdots+b_{p-1}) + qk(b_{0}+b_{1}+\cdots+b_{p-1}) \\
            &= (b_{0}+b_{1}+\cdots+b_{p-1}).
  \end{align*}
  The conclusion now promptly follows from the definition of the~$x_i$, establishing the proposition.
\end{proof}

Assume that~$q$ is a prime.
In this case~$H_{1}(\Sigma_{p}(T(p,q)))$ becomes an~$\F_q$-vector space.
The covering action~$t$ is then an~$\F_q$-linear endomorphism of \(V_{p,q}\).

\section{Twisted polynomials of torus knots}
\label{sec:TwistedPolynomial}

In this this section, we compute the  Alexander polynomial of the~$0$-framed surgery~$M_{T(p,q)}$ twisted by a metabelian representation~$\alpha_{T(p,q)}(p,\chi) \colon \pi_1(M_{T(p,q)}) \to GL_p(\LC)$ that frequently appears in Casson-Gordon theory~\cite{HeraldKirkLivingston}.
In Subsection~\ref{sub:MetabRep}, we recall the definition of~$\alpha_K(p,\chi)$ for a general knot~$K$, in Subsection~\ref{sub:MetabForTorusKnot}, we restrict to torus knots, and in Subsection~\ref{sub:TwistedPolynomial}, we compute the relevant twisted Alexander polynomials.

\subsection{The metabelian representation~$\alpha_K(p,\chi)$}
\label{sub:MetabRep}
In this subsection, given a knot~$K$ and a positive integer~$p$, we recall the definition of the representation~$\alpha_K(p,\chi)\colon \pi_1(M_K) \to GL_p(\LC)$ from~\cite{HeraldKirkLivingston}.
In what follows,~$E_K$ denotes the exterior of~$K$ and~$M_K$ denotes its~$0$-framed surgery.
Finally, we use~$\xi_m:=e^{2\pi i/m}$ to denote the~$m$-th primary root of unity.
\medbreak
We use~$H_1(E_K;\Z[t_{K}^{\pm1}])\cong \pi_1(E_K)^{(1)}/\pi_1(E_K)^{(2)}$ to denote the Alexander module of~$K$.
In what follows, we shall frequently identify~$H_1(\Sigma_p(K))$ with~$H_1(E_K;\Z[t_K^{\pm 1}])/(t_K^p-1)$, as for instance in~\cite[Corollary 2.4]{FriedlEta}.
Consider the following composition of canonical projections:
\begin{equation}
  \label{eq:qK}
  q_K \colon \pi_{1}(M_{K})^{(1)}  \to  H_1(E_K;\Z[t_K^{\pm 1}]) \to H_1(\Sigma_p(K)).
\end{equation}
Use~$\phi_K \colon \pi_1(E_K) \to  H_1(E_K;\Z)\cong \Z=\langle t_K \rangle$  to denote the abelianization homomorphism, and fix an element~$\mu_{K}$ in~$\pi_1(E_K)$ such that~$\phi_K(\mu_{K})=t_K$.
Note that for every~$g \in \pi_1(E_K)$, we have~$\phi_K(\mu_K^{-\phi_K(g)}g)=1$.
Since~$\phi_K$ is the abelianization map, we deduce that~$\mu_K^{-\phi_K(g)}g$ belongs to~$\pi_1(E_K)^{(1)}$.
Combining these notations, we consider the following representation:

\begin{align}
  \label{eq:Matrix}
  \alpha_K(p,\chi) &\colon \pi_1(E_K) \to ~\operatorname{GL}_p(\LC)  \nonumber \\
  \alpha_K(p,\chi)(g)&= \begin{pmatrix}
    0& 1 & \cdots &0 \\
    \vdots & \vdots & \ddots & \vdots  \\
    0 & 0 & \cdots & 1 \\
    t & 0 & \cdots & 0
  \end{pmatrix}^{\phi_K(g)}
                     \begin{pmatrix}
                       \xi_{m}^{\chi(q_K(\mu_K^{-\phi_K(g)}g))} & 0 & \cdots &0 \\
                       0 & \xi_{m}^{\chi(t_K \cdot q_K(\mu_K^{-\phi_K(g)}g))} & \cdots &0 \\
                       \vdots & \vdots & \ddots & \vdots \\
                       0 & 0 & \cdots & \xi_{m}^{\chi(t_K^{p-1} \cdot q_K(\mu_K^{-\phi_K(g)}g))}
                     \end{pmatrix} \nonumber \\
                   &=:A_p(t)^{\phi_K(g)}\operatorname{diag}\left(\xi_m^{\chi(q_K(\mu_K^{-\phi_K(g)}g))},\ldots,\xi_m^{\chi(t_K^{p-1} \cdot q_K(\mu_K^{-\phi_K(g)}g))}\right).
\end{align}
Note that~$\alpha(p,\chi)$ can equally well be defined on~$\pi_1(M_K)$ instead of~$\pi_1(E_K)$: the definition can be adapted \textit{verbatim}, and we use the same notation:
$$   \alpha_K(p,\chi) \colon \pi_1(M_K) \to ~\operatorname{GL}_p(\LC). ~$$
A closely related observation is that~$\alpha(p,\chi)$ is a metabelian representation and therefore vanishes on the longitude of~$K$; this also explains why~$\alpha_K(p,\chi)$ descends to ~$\pi_1(M_K)~$.

\subsection{An explicit description of~$\alpha_{T(p,q)}(p,\chi)$.}
\label{sub:MetabForTorusKnot}
We use the presentation of~$\pi_1(E_{T(p,q)})$ from Proposition~\ref{prop:complement-torus-knot} to describe the representation~$\alpha_{T(p,q)}(p,\chi)$.
In this subsection, we will often set~$K:=T(p,q)$ in order to avoid cumbersome notations such as~$q_{T(p,q)}$.
\medbreak
We recall the definition of the generators~$x_0,\ldots,x_{p-1}$ of~$H_1(\Sigma_p(K)) \cong \Z_q^{p-1}$ described in Proposition~\ref{prop:homology-group-cover}, referring to Section~\ref{sec:branch-covers-torus} for further details.
Using the notations of that section, we set~$x_i=b_i-\mu^p$, where~$\mu$ is a meridian of~$K$.
Thinking of~$x_i$ as the abelianisation of~$\mu^{-p}b_i$, and using Proposition~\ref{prop:FundamentalGroupCyclicCover} to identify~$b_i$ with~$\mu^ic_2\mu^{-i}$, we have 
\begin{equation}
  \label{eq:PracticalForCharacter}
  t_K^iq_K(\mu^{-p}c_2)
  =q_K(\mu^{-p}\mu^i c_2\mu^{-i})
  =q_K(\mu^{-p}b_i)
  =x_i.
\end{equation}
Recall furthermore that Proposition~\ref{prop:homology-group-cover} also established the relations~$x_0+\cdots+x_{p-1}=0$ as well as~$t_Kx_i=x_{i+1}$. 
The next result follows immediately from these considerations.

\begin{lemma}
  \label{lem:Characaters}
  Let~$p,q>0$ be two coprime integers.
  The abelian group of characters on \(H_1(\Sigma_{p}(T(p,q)))\) is isomorphic to
  \[\{ \mathbf{a}:=(a_{1},\ldots,a_{p}) \in \Z_{q}^{p} \ | \ a_1+\cdots+a_p=0\}.\]
  The isomorphism maps a character~$\chi$ to~$(\chi(x_0),\ldots,\chi(x_{p-1}))$, and we write~$\chi_{\mathbf{a}}$ for the character associated to~$\mathbf{a}$.
\end{lemma}

Recall that Proposition~\ref{prop:complement-torus-knot} described a two-generator one-relation presentation for the knot group~$\pi_1(E_{T(p,q)})$: the generators were denoted by~$c_1$ and~$c_2$, and the unique relator was~$c_1^pc_2^{-q}$.
The next proposition describes the image of these generators under ~$\alpha(p,\chi):=\alpha_{T(p,q)}(p,\chi)$.
This will be useful in Proposition~\ref{prop:TwistedAlexanderPolynomial} when we compute the twisted Alexander polynomial of~$E_{T(p,q)}$.

\begin{proposition}
  \label{prop:alphac1c2}
  Let~$p,q>0$ be two coprime integers.
  For a character \(\chi=\chi_{\mathbf{a}}\) on \(H_1(\Sigma_p(T(p,q)))\), the representation~$\alpha(p,\chi)$ is conjugated to a representation~$\alpha'(p,\chi)$ such that 
  \begin{align*}
    \alpha'(p,\chi)(c_{2}) &= t \cdot \diag(\xi_q^{a_{1}},\ldots,\xi_q^{a_{p}}), \nonumber \\
    \alpha'(p,\chi)(c_{1}) &= A_{p}(t)^{q}.
  \end{align*}
\end{proposition}
\begin{proof}
  We first compute~$\alpha(p,\chi)(c_{2})$.
  We know that~$\phi_K(c_2)=p$ and~$A_p(t)^p=t \cdot \id$.
  In order to compute the diagonal matrix which appears in the definition of~$\alpha(p,\chi)(c_{2})$ (recall~\eqref{eq:Matrix}), we use~\eqref{eq:PracticalForCharacter} and Lemma~\ref{lem:Characaters} to obtain~$\chi(t_K^{i-1} q_K(\mu^{-p}c_2))=\chi(x_{i-1})=a_i$.
  The first assertion follows:
  \[\alpha(p,\chi)(c_{2}) 
    = t \cdot \diag(\xi_q^{\chi(q_K(\mu^{-p}c_2))},\ldots,\xi_q^{\chi(t_K^{p-1}q_K(\mu^{-p}c_2))})=t \cdot \diag(\xi_q^{a_{1}},\ldots,\xi_q^{a_{p}}).\]
  Next, we study the conjugacy class of~$\alpha(p,\chi)(c_{1})$: we must find an invertible matrix \(X\) such that
  \begin{align}
    X \alpha(p,\chi)(c_{1}) X^{-1} &= A_{p}(t)^{q}, \label{eq:conjugation-c1} \\
    X \alpha(p,\chi)(c_{2}) X^{-1} &= t \cdot \diag(\xi_q^{a_{1}},\ldots,\xi_q^{a_{p}}). \label{eq:conjugation-c2}
  \end{align}
  For \(v \in H_{1}(\Sigma_{p}(K))\), we define~$\widetilde{\alpha}(v) := \diag(\xi_{q}^{\chi(v)},\xi_{q}^{\chi(t_{K}v)},\ldots,\xi_{q}^{\chi(t_{K}^{p-1}v)})~$.
  Observe that if we set~\(X := \widetilde{\alpha}(z)\), then~\eqref{eq:conjugation-c2} is satisfied for any \(z \in H_{1}(\Sigma_{p}(K))\): indeed~$\alpha(p,\chi)(c_2)$ commutes with~$X$ since both are diagonal.
  Therefore, we just have to establish the existence of a \(z \in H_{1}(\Sigma_{p}(K))\) such that~\eqref{eq:conjugation-c1} is satisfied for~$X= \widetilde{\alpha}(z)$.
  
  First, for any \(x \in H_{1}(\Sigma_{p}(K))\) a computation shows that the following equation holds:
  \[\widetilde{\alpha}(x) A_{p}(t)^{q} \widetilde{\alpha}(x)^{-1} = A_{p}(t)^{q} \widetilde{\alpha}((t_{K}^{-q}-1)x).\]
  Define \(y := q_{K}(\mu^{-q}c_{1}) \) so that~$\alpha(p,\chi)(c_{1}) = A_{p}(t)^{q} \widetilde{\alpha}(y).$
  Consequently, if we set~$X:=\widetilde{\alpha}(z)$ (for any~$z \in H_1(\Sigma_p(K))$), use the definition of~$y$, the fact that~$\widetilde{\alpha}(y)$ and~$X$ commute (both are diagonal), and the aforementioned identity, then we obtain
  \begin{align*}
    X \alpha(p,\chi)(c_{1}) X^{-1} 
    =X  A_{p}(t)^{q} \widetilde{\alpha}(y) X^{-1} 
    &=X  A_{p}(t)^{q}X^{-1} \widetilde{\alpha}(y)  \\
    &= A_{p}(t)^{q} \widetilde{\alpha}((t_{K}^{-q}-1)z+y).
  \end{align*}
  Therefore, if we choose \(z := -(t_{K}^{-q}-1)^{-1}y\), then~\eqref{eq:conjugation-c1} holds.
  For this to make sense however, we must argue that~$(t_{K}^{-q}-1)$ is an automorphism of~$H_{1}(\Sigma_{p}(K))$.
    This is indeed the case: as~$t_K-1$ is an automorphism of~$H_{1}(\Sigma_{p}(K))$, the inverse is given by~$(t_{K}^{-1}-1)^{-1} (1+t_{K}^{-q}+t_{K}^{-2q}+\cdots+t_{K}^{-(k-1)q})$, where~$qk \equiv 1$ mod~$p$. Such a~$k$ exists because~$p$ and~$q$ are coprime.
  We have therefore found~$X$ such that~\eqref{eq:conjugation-c1} and~\eqref{eq:conjugation-c2} hold, and this concludes the proof of the proposition.
\end{proof}

\subsection{The computation of the twisted polynomial}
\label{sub:TwistedPolynomial}

In this subsection, we compute the twisted Alexander polynomial of the~$0$-framed surgery~$M_{T(p,q)}$ with respect to~$\alpha(p,\chi)$.
\medbreak
Recall that given a space~$X$ and a representation~$\rho \colon \pi_1(X) \to GL_p(\LC)$, the \emph{twisted Alexander polynomia}l~$\Delta_{1}^{\rho}(X)$ is defined as the order of the twisted Alexander module~$H_{1}(X; \C[t^{\pm1}]^{p}_{\rho})$.
More generally, we write~$\Delta_{i}^{\rho}(X)$ for the order of the~$\LC$-module~$H_{i}(X; \C[t^{\pm1}]^{p}_{\rho})$.
Recall that the~$\Delta_{i}^{\rho}(X)$ are defined up to multiplication by units of~$\LC$.

The next proposition describes~$\Delta_{1}^{\alpha(p,\chi)}(E_{T(p,q)})$, where~$E_{T(p,q)}$ denotes the exterior of~$T(p,q)$.

\begin{proposition}
  \label{prop:TwistedAlexanderPolynomial}
  Let~$p,q>0$ be coprime integers.
  For~$\chi=\chi_{\mathbf{a}} \colon H_1(\Sigma_p(T(p,q))) \to~GL_p(\LC)$, the metabelian twisted Alexander polynomial of~$E_{T(p,q)}$ is given by
  ~$$\Delta_{1}^{\alpha(p,\chi)}(E_{T(p,q)})= \frac{(1-t^{q})^{p-1}}{(t\xi_q^{a_{1}}-1)(t\xi_q^{a_{2}}-1) \cdots (t\xi_q^{a_{p}}-1)}.~$$  
\end{proposition}
\begin{proof}
  We use~$\tau^{\alpha(p,\chi)}(E_{K})$ to denote the Reidemeister torsion of a knot exterior~$E_{K}$ twisted by~$\alpha(p,\chi):=\alpha_K(p,\chi)$.
  We refer to~\cite{FriedlVidussiSurvey} for further references on the subject, but simply note that~$\tau^{\alpha(p,\chi)}(E_{K})$ is defined since the chain complex~$C_*(E_{K};\C(t)^p)$ of left~$\C(t)$-modules is acyclic~\cite[Corollary after Lemma~4]{CassonGordon1}. 
  Since~$E_{K}$ has torus boundary, by \cite[Proposition 2, item 5]{FriedlVidussiSurvey}, the twisted Reidemeister torsion and twisted Alexander polynomial are related~by
  ~$$\tau^{\alpha(p,\chi)}(E_{K})=\frac{\Delta_1^{\alpha(p,\chi)}(E_{K})}{\Delta_0^{\alpha(p,\chi)}(E_{K})}.$$ 
  Since~$\Delta_{0}^{\alpha(p,\chi)}(E_{K})= 1$ for every knot~$K$~\cite[Lemma 8.1]{BorodzikConwayPolitarczyk}, we are reduced to computing~$\tau^{\alpha(p,\chi)}(E_{T(p,q)}).$
    By~\cite[Theorem A]{Kitano}, this torsion invariant can be expressed via Fox calculus.
  In our case, using the presentation of~$\pi_1(E_{T(p,q)})$ resulting from Proposition~\ref{prop:complement-torus-knot}, we obtain
  \begin{equation}
    \label{eq:AlexanderFox}
    \Delta_{1}^{\alpha(p,\chi)}(E_{T(p,q)}) =\tau^{\alpha(p,\chi)}(E_{T(p,q)})= \frac{\det\left(\alpha(p,\chi)\left( \frac{\partial (c_{1}^{p}c_{2}^{-q})}{\partial c_{1}} \right)\right)}{\det\left( \alpha(p,\chi)(c_{2})-\id) \right)}.
  \end{equation}
  Since this expression does not depend on the conjugacy class of~$\alpha(p,\chi)$, we can work with the representation~$\alpha'(p,\chi)$ described in Proposition~\ref{prop:alphac1c2}.
  Using the first item of Proposition~\ref{prop:alphac1c2}, the denominator of~\eqref{eq:AlexanderFox} is given by the formula
  \begin{equation}
    \label{eq:Denominator}
    \det(\alpha(p,\chi)(c_{2})-\id) = \det (\diag(t \xi_q^{a_{1}}-1,t \xi_q^{a_{2}}-1,\ldots,t \xi_q^{a_{p}}-1)) = \prod_{i=1}^{p}(t\xi_q^{a_{i}}-1).
  \end{equation}
  We will now compute the numerator of~\eqref{eq:AlexanderFox} and show that it equals~$(1-t^q)^{p-1}$.
  Recall from~\eqref{eq:Matrix} that for~$g \in \pi_1(E_K)$, the metabelian representation~$\alpha_K(p,\chi)$ is given by~$\alpha_K(p,\chi)(g)=A_p(t)^{\phi_K(g)}D_g$, where~$D_g$ is the diagonal matrix with~$\xi_q^{\chi(t_K^{i-1} \cdot q_K(\mu_{K}^{-\phi_K(g)}g))}$ as its~$i$-th diagonal component.
  An inductive argument involving the properties of the Fox derivative shows that
  \begin{align*}
    \frac{\partial (c_{1}^{p}c_{2}^{-q})}{\partial c_{1}} &= \frac{\partial c_{1}^{p}}{\partial c_{1}} = 1+c_{1}+c_{1}^{2}+\cdots+c_{1}^{p-1}=:g.
  \end{align*}
  We will now apply~$\alpha(p,\chi)$ to~$g$.
  We recall from Proposition~\ref{prop:alphac1c2} that~$\alpha'(p,\chi)(c_1)=A_p(t)^q$, and we now work over~$\C[t^{\pm 1/p}]$.
  Indeed, as observed in~\cite[page 935]{HeraldKirkLivingston}, in this ring, the matrix~\(A_{p}(t)\) is conjugated to the diagonal matrix
  \[B_{p}(t) := \diag(t^{1/p},\xi_{p}t^{1/p},\xi_{p}^{2}t^{1/p},\ldots,\xi_{p}^{p-1}t^{1/p}).\]
  Since~\eqref{eq:AlexanderFox} only depends on the conjugacy class of the representation~$\alpha(p,\chi)$, we can work with~$B_p(t)$ instead of~$A_p(t)$.
  We use~$\sim$ to denote the conjugacy relation.
  Since~$B_p(t)$ is diagonal, its powers are easy to compute, and as a consequence, we obtain
  \begin{align*}
    \alpha'(p,\chi)(g) 
      &\sim \id + B_{p}(t)^{q} + B_{p}(t)^{2q} + \cdots + B_{p}(t)^{(p-1)q} \\
      &= \diag\left( \frac{1-t^{q}}{1-t^{q/p}}, \frac{1-t^{q}}{1-\xi_{p}^{q}t^{q/p}}, \frac{1-t^{q}}{1-\xi_{p}^{2q}t^{q/p}}, \ldots, \frac{1-t^{q}}{1-\xi_{p}^{q(p-1)}t^{q/p}} \right).
  \end{align*}
  Taking the determinant of this expression, we deduce that
  \begin{align}
    \label{eq:Numerator}
    \det \left(\alpha(p,\chi)\left( \frac{\partial (c_{1}^{p}c_{2}^{-q})}{\partial c_{1}} \right)\right) = \prod_{j=0}^{p-1}\frac{1-t^{q}}{1-\xi_{p}^{jq}t^{q/p}} = \frac{(1-t^{q})^{p}}{1-t^{q}} = (1-t^{q})^{p-1}.
  \end{align}
  Plugging~\eqref{eq:Denominator} and~\eqref{eq:Numerator} into~\eqref{eq:AlexanderFox} concludes the proof of the proposition.
\end{proof}

Using Proposition~\ref{prop:TwistedAlexanderPolynomial}, we can compute the twisted polynomial of the~$0$-framed surgery~$M_{T(p,q)}$.

\begin{corollary}
  \label{cor:TwistedAlexanderPolynomial}
  Let~$p,q>0$ be two coprime integers.
  For~$\chi=\chi_{\mathbf{a}} \colon H_1(\Sigma_p(T(p,q))) \to GL_p(\LC)$, the metabelian twisted Alexander polynomial of~$M_{T(p,q)}$ is given by
  ~$$\Delta_{1}^{\alpha(p,\chi)}(M_{T(p,q)}))= \frac{(-1)^{p-1}(1-t^{q})^{p-1}}{(t\xi_q^{a_{1}}-1)(t\xi_q^{a_{2}}-1) \cdots (t\xi_q^{a_{p}}-1)(t-1)}.~$$
\end{corollary}
\begin{proof}
  By Proposition~\ref{prop:TwistedAlexanderPolynomial}, we need only show that~$(-1)^{p-1}(t-1)\Delta_{1}^{\alpha(p,\chi)}(M_{K})=\Delta_{1}^{\alpha(p,\chi)}(E_{K})$ for every knot~$K$, where~$\alpha(p,\chi):=\alpha_K(p,\chi)$.
  Using the equality~$\Delta_{1}^{\alpha(p,\chi)}(E_{K})
  =\tau^{\alpha(p,\chi)}(E_{K})$ that was obtained 
  in the proof of Proposition~\ref{prop:TwistedAlexanderPolynomial},~\cite[Lemma~3]{FriedlVidussiSurvey}, as well as~\cite[Proposition 2, item~(8)]{FriedlVidussiSurvey},~\cite[Proposition 5]{FriedlVidussiSurvey}, and the fact that~$\Delta_{0}^{\alpha(p,\chi)}(M_{K})=1$ (by~\cite[Lemma~8.1]{BorodzikConwayPolitarczyk}), we obtain the following sequence of equalities:
  \begin{align*}
    \Delta_{1}^{\alpha(p,\chi)}(E_{K})
    &=\tau^{\alpha(p,\chi)}(E_{K})
      =\det(\alpha(p,\chi)(\mu_K)-\id)\tau^{\alpha(p,\chi)}(M_{K}) \\
    &=\det(\alpha(p,\chi)(\mu_K)-\id)\frac{\Delta_{1}^{\alpha(p,\chi)}(M_{K})}{\Delta_{0}^{\alpha(p,\chi)}(M_{K})\Delta_{2}^{\alpha(p,\chi)}(M_{K})}\\
    &=\det(\alpha(p,\chi)(\mu_K)-\id)\frac{\Delta_{1}^{\alpha(p,\chi)}(M_{K})}{\Delta_{0}^{\alpha(p,\chi)}(M_{K})\overline{\Delta_{0}^{\alpha(p,\chi)}(M_{K})}} \\
    &=\det(\alpha(p,\chi)(\mu_K)-\id)\Delta_{1}^{\alpha(p,\chi)}(M_{K}).
  \end{align*}
  It thus remains to show that~$\det(\alpha(p,\chi)(\mu_K)-\id)=(-1)^{p-1}(t-1)$: this follows from the definition of~$\alpha(p,\chi)$ (recall~\eqref{eq:Matrix}) since~$\alpha(p,\chi)(\mu_K)=A_p(t)$.
  This concludes the proof of the proposition.
\end{proof}

\section{Linking forms and their metabolisers}
\label{sec:Metabolisers}

This section collects some facts about linking forms and their metabolizers.
This will be useful in Section~\ref{sec:MainTheorem} since both the metabelian Blanchfield pairing and $\lambda_p(T(p,q))$ are linking forms.
In Subsection~\ref{sub:LinkingForms}, we recall some basics on linking forms and their Witt groups.
In Subsection~\ref{sub:Graph}, we prove a result on metabolisers of linking forms of the type~$(V_1 \oplus V_2,\lambda_1 \oplus -\lambda_2)$.

\subsection{The Witt group of linking forms}
\label{sub:LinkingForms}
Let~$R$ be a PID with involution, and let~$Q$ denote its field of fractions.
This subsection is concerned with linking forms.
Firstly, we recall the definition of the Witt group~$W(Q,R)$ of linking forms.
Secondly, we collect some facts about~$W(\C(t),\LC)$ that are used in Section~\ref{sec:MainTheorem} below.
\medbreak

A \emph{linking form} over~$R$ is a pair~$(V,\lambda)$, where~$V$ is a torsion~$R$-module, and~$\lambda \colon V \times V \to~Q/R$ is a sesquilinear and Hermitian pairing.
A linking form~$(V,\lambda)$ is \emph{non-singular} if its \emph{adjoint}~$\lambda^\bullet \colon V \to~V^*, \linebreak x \mapsto \lambda(x,-)$ is an isomorphism.
In the sequel, our linking forms will be either over~$\Z$ or~$\C[t^{\pm 1}]$.
From now on, we also assume that all linking forms are non-singular.
Given a linking form~$(V,\lambda)$ over~$R$, a submodule~$L \subset V$ is \emph{isotropic} if~$L \subset L^\perp$ and is a \emph{metaboliser} if~$L=L^\perp$. 
A linking form is \emph{metabolic} if it admits a metabolizer.

\begin{definition}
  \label{def:metabolic_and_so_on}
  The \emph{Witt group of linking forms}, denoted~$W(Q,R)$, consists of the monoid of linking forms modulo the submonoid of metabolic linking forms.
  Two linking forms~$(V,\lambda)$ and~$(V',\lambda')$ are called \emph{Witt equivalent} if they represent the same element in~$W(Q,R)$.
\end{definition}

The Witt group of linking forms is known to be an abelian group under direct sum, where the inverse of the class~$[(V,\lambda)]$ is represented by~$(V,-\lambda)$.
Next, we collect some facts on~$W(\C(t),\LC)$ that will be used in Section~\ref{sec:MainTheorem} below.

\begin{remark}
  \label{rem:JumpsRoots}
  The Witt group~$W(\C(t),\C[t^{\pm 1}])$ is known to be free abelian and is detected by the signature jumps~$\delta \sigma_{(V,\lambda)}$~\cite[Sections 4 and~5]{ BorodzikConwayPolitarczyk}.
  In particular, a linking form~$(V,\lambda)$ over~$\LC$ is metabolic if and only if all its signature jumps vanish~\cite[Theorem~5.3]{ BorodzikConwayPolitarczyk}.
  Reformulating,~$[V,\lambda]=0$ in~$W(\C(t),\C[t^{\pm 1}])$ if and only if~$\delta \sigma_{(V,\lambda)}(\omega)=0$ for all~$\omega \in S^1$.
  We refer to~\cite[Sections 4 and~5]{ BorodzikConwayPolitarczyk} for further details regarding signatures of linking forms but note that a linking form~$(V,\lambda)$ will have a trivial jump at~$\omega \in S^1$ if the order~$\operatorname{Ord}(T)$ of the~$\LC$-module~$T$ does not have a root at~$\omega$.
\end{remark}

In particular, Remark~\ref{rem:JumpsRoots} implies the following result about linear independence in~$W(\C(t),\LC)$.

\begin{proposition}
  \label{prop:Splitting}
  If~$(V_1,\lambda_1)$ and~$(V_2,\lambda_2)$ are two linking forms over~$\LC$ such that~$\operatorname{Ord}(V_1)$ and~$\operatorname{Ord}(V_2)$ have distinct roots, then the following assertions hold:
  \begin{enumerate}
  \item if~$(V_1,\lambda_1)$ and~$(V_2,\lambda_2)$ are not metabolic, then the Witt classes~$[V_1,\lambda_1]$ and~$[V_2,\lambda_2]$ are linearly independent in \(W(\C(t),\C[t^{\pm 1}])\);
  \item if~$(V_1,\lambda_1) \oplus (V_2,\lambda_2)$ is metabolic, then~$(V_1,\lambda_1)$ and~$(V_2,\lambda_2)$ are both metabolic.
  \end{enumerate}
\end{proposition}
\begin{proof}
  We only prove the first assertion as the second assertion follows immediately.
  Assume that~$n_1 [V_1,\lambda_1]+n_2 [V_2,\lambda_2]=0$ for some integers~$n_1$ and~$n_2$.
  Remark~\ref{rem:JumpsRoots} implies that all the signature jumps of~$n_1 \lambda_1 \oplus n_2 \lambda_2$ must vanish.
  Since~$\lambda_1$ is not metabolic, Remark~\ref{rem:JumpsRoots} also implies that~$\lambda_1$ admits a non-trivial signature jump at some~$\omega_1 \in S^1$.
  As a consequence of these two assertions, we infer that~$n_1\lambda_1$ and~$n_2 \lambda_2$ must have a non-trivial signature jump at~$\omega_1$.
  Since~$\operatorname{Ord}(V_1)$ and~$\operatorname{Ord}(V_2)$ have distinct roots, we deduce that~$n_1=0$.
  The same reasoning shows that~$n_2=0$, thus establishing the linear independence of~$[V_1,\lambda_1]$ and~$[V_2,\lambda_2]$ and establishing the proposition.
\end{proof}

\subsection{Graph metabolisers}
\label{sub:Graph}
Given linking forms~$(V_1,\lambda_1),(V_2,\lambda)$, we prove a result on metabolisers of linking forms of the type~$(V_1 \oplus V_2,\lambda_1 \oplus -\lambda_2)$.
More precisely, Proposition~\ref{prop:DirectSummetaboliser} provides a criterion for when such a metabolizer must be a graph.
This result will be used in Section~\ref{sec:MainTheorem} when we study metabolisers of~$\lambda_p(T(p,q))^N \oplus -\lambda_p(T(p,q))^N.$
\medbreak
Given linking forms~$(V_1,\lambda_1)$ and~$(V_2,\lambda_2)$, a \emph{morphism} of linking forms is an~$R$-linear homomorphism~$f \colon V_1 \to V_2$ such that~$\lambda_2(f(x),f(y))=\lambda_1(x,y)$ for all~$x,y \in V_1$.
Observe that if the forms are non-singular, then a morphism is necessarily injective.
An \emph{isometry} of linking forms is a bijective morphism of linking forms.
The graph 
$$  \Gamma_f= \{(v,f(v)) \in V_1 \oplus V_2 \ | \ v \in V_{1}\}~$$ 
of a morphism~$f \colon (V_1,\lambda_1) \to (V_2,\lambda_2)$ is an isotropic submodule of~$(V_1 \oplus V_2,\lambda_1 \oplus -\lambda_2)$.
If~$f$ is an isometry, then~$ \Gamma_f$ is in fact a metaboliser of ~$(V_1 \oplus V_2,\lambda_1 \oplus -\lambda_2)$. 
The next proposition provides an assumption under which the converse also holds.

\begin{proposition}
  \label{prop:DirectSummetaboliser}
  Let~$(V_1,\lambda_1)$ and~$(V_2,\lambda_2)$ be linking forms over~$R$, and let~$L \subset V_1 \oplus V_2$ be a metaboliser of~$\lambda_1 \oplus -\lambda_2$. 
  The following assertions hold:  
  \begin{enumerate}
  \item  if \(L \cap (V_{1} \oplus 0) =0= L \cap (0 \oplus V_{2})\), then ~$L$ is the graph of an isometry~$f \colon V_{1} \to V_{2}$:
    \[L = \{(v,f(v)) \in V_1 \oplus V_2 \ | \ v \in V_{1}\} ;\]
  \item if we additionally work over~$R=\Z$, suppose that \(V_{1}\) and \(V_{2}\) are equipped with an isometric~$\Z_p-$action, and~\(L\) is a \(\Z_{p}\)-invariant metaboliser, then the isometry \(f\) is \(\Z_{p}\)-equivariant.
  \end{enumerate}
\end{proposition}
\begin{proof}
  We prove the first assertion.
  The isometry~$f$ will be defined by using the canonical projections \(\pr_{i} \colon V_1 \oplus V_2 \to V_{i}\)  for \(i=1,2\).
  Since \(L \cap (V_{1} \oplus 0) =0= L \cap (0 \oplus V_{2})\), it follows that~\(\pr_{i}|_{{L}}\) is injective, for \(i=1,2\).
  Set \(W_{i} := \pr_{i}(L)\), for \(i=1,2\), and define~$f$ as the composition
  \[f \colon W_{1} \xrightarrow{\pr_{1}^{-1},\cong} L \xrightarrow{\pr_{2},\cong} W_{2}.\]
  Since~$f$ is an isomorphism of~$R$-modules, it remains to check that it is a morphism of linking forms. 
  First however, we use the definition of~$f$ to observe that
  \begin{equation}
    \label{eq:NearlyL}
    L = \{(v,f(v)) \in V_1 \oplus V_2 \ | \ v \in W_{1}\} \subset V_1 \oplus V_2.
  \end{equation}
  The fact that~$f$ is a morphism now follows from the fact that~$L$ is isotropic: for any~$v,w \in~W_1$, the pairs~$(v,f(v)), (w,f(w))$ belong to~$L$, and therefore we have
  \[0 = (\lambda_{1} \oplus -\lambda_{2})((v,f(v)),(w,f(w))) = \lambda_{1}(v,w) - \lambda_{2}(f(v),f(w)).\]
  Looking at~\eqref{eq:NearlyL}, it only remains to show that~$V_1=W_1$ and~$V_2=W_2$.
  Since~$f$ is an isomorphism, we have  \(\ord(W_{1}) = \ord(W_{2})\) and therefore~\eqref{eq:NearlyL} implies that~$ \ord(L)^{2} = \ord(W_{1}) \ord(W_{2})$.
  Since~\(L\) is a metaboliser, we deduce that
  \begin{equation}
    \label{eq:Proportional}
    \ord(V_{1}) \ord(V_{2}) = \ord(L)^{2} = \ord(W_{1}) \ord(W_{2}).
  \end{equation}
  By way of contradiction, assume that~\(\ord(W_{1}) \) divides \( \ord(V_{1})\), but that \(\ord(W_{1}) \neq \ord(V_{1})\); we write \(\ord(W_{1}) \nmid \ord(V_{1})\).
  A glance at~\eqref{eq:Proportional} shows that~\(\ord(V_{2})~\nmid~\ord(W_{2})\), contradicting the inclusion~$W_2 \subset V_2$. 
  We conclude that \(\ord(W_{i}) = \ord(V_{i})\) and consequently~\(W_{i} =~V_{i}\), for \(i=1,2\).
  This concludes the proof of the first assertion.

  We prove the second assertion.
  Use~$t$ to denote a generator of~$\Z_p$. 
  As the metaboliser~$L$ is~$\Z_p$-invariant, observe that if~\((v,f(v)) \in L\), then~\((tv,tf(v)) \in L\) for any~$v \in V_1$.
  Moreover, as~\((tv,f(tv)) \in L\) and \( L \cap (0 \oplus V_{2})=0\), it follows that~\((tv,f(tv) )= (tv,tf(v))\).
  We have therefore established that \(f(tv)=tf(v)\) for any~$v \in V_1$, and thus~$f$ is~$\Z_p$-equivariant, as desired.
  This concludes the proof of the proposition.
\end{proof}

\section{Non-slice linear combinations of iterated torus knots }
\label{sec:MainTheorem}

This section aims to prove Theorem~\ref{thm:Main} from the introduction, whose statement we now recall.
For an integer \(p \geq 2\) and a sequence \(Q = (q_{1},q_{2},\ldots,q_{\ell})\) of integers that are relatively prime to~$p$, we use the following notation for iterated torus knots:~$T(p,Q):= T(p,q_{1};p,q_{2};\ldots;p,q_{\ell}).$
Our main result reads as follows.

\begin{theorem}
  \label{thm:LinIndep}
  Fix a prime power~$p$. 
  Let \(\mathcal{S}_{p}\) be the set of iterated torus knots \(T(p,q_{1};p,q_{2};\ldots;p,q_{\ell})\), where the sequences~$(q_{1},q_{2},\ldots,q_{\ell})$ of positive integers that are coprime to~$p$ satisfy
  \begin{enumerate}
  \item $q_\ell$ is a prime;
  \item for \(i=1,\ldots,\ell-1\), the integer \(q_{i} \) is coprime to \(q_{\ell}\) when~$\ell >1$;
  \end{enumerate}
  The set \(\mathcal{S}_{p}\) is linearly independent in the topological knot concordance group~$\mathcal{C}^{\text{top}}$.
\end{theorem}

To prove Theorem~\ref{thm:LinIndep}, we must obstruct the sliceness of linear combinations of knots belonging to~$\mathcal{S}_p$.
The first step, which is carried out in Subsection~\ref{sub:algebr-slice-line}, is to determine which of these linear combinations are algebraically slice.
In Subsection~\ref{sub:free-subgr-gener}, we use metabelian twisted Blanchfield pairings to obstruct the sliceness of such algebraically slice linear combinations.

\subsection{Algebraically slice linear combinations of algebraic knots}
\label{sub:algebr-slice-line}

Fix an integer~$p \geq 2$.
For~$i=1,\ldots, k$, fix sequences \(Q_{i} = (q_{i,1},q_{i,2},\ldots,q_{i,\ell_i})\) of \(\ell_{i}\) positive integers each of which is coprime to~$p$, and let \(n_{1},\ldots,n_{k} \in \Z\).
The goal of this subsection is to determine when the following knot is algebraically slice:
\begin{equation}
  \label{eq:LinearCombination}
  K = n_{1}T(p,Q_{1}) \# n_{2} T(p,Q_{2}) \# \cdots \# n_{k} T(p,Q_{k}).
\end{equation}
In order to provide a convenient criterion, we define the \(s\)-\emph{level} of \(K\) to be the following knot:
$$ \mathcal{K}_{s}(K) :=n_{1} T(p,q_{1,\ell_{1}-s}) \# n_{2} T(p,q_{2,\ell_{2}-s}) \# \ldots \# n_{k} T(p,q_{k,\ell_{k}-s}).
$$
Here, it is understood that \(T(p,q_{i}^{\ell_{i}-s})\) is the unknot \(U\) if \(\ell_{i}-s < 1\).
As an example of this notation, we see that if~$Q = (q_{1},\ldots,q_{\ell})$, then~$\mathcal{K}_{s}(T(p,Q)) =T(p,q_{\ell-s})$ for~$0 \leq s \leq \ell-1$ and~$\mathcal{K}_{s}(T(p,Q)) = U$, for~$s \geq \ell.$
In particular, the cabling formula for the classical Blanchfield form implies that
\begin{equation}
  \label{eq:DecompositionTorusKnot}
  \Bl(T(p,Q)) \cong \bigoplus_{s \geq 0}\Bl(\mathcal{K}_{s}(T(p,Q)))(t^{p^{s}}).
\end{equation}
Indeed, for a knot~$L$, the cabling formula reads as~$\Bl(L_{p,q})(t)=\Bl(T(p,q))(t)\oplus\Bl(L)(t^p)$~\cite{LivingstonMelvin}.
Next, we move on to a slightly more involved example.

\begin{example}
  \label{ex:Decompo}
  The~$s$-levels of \(J := T(p,q_{1};p,q_{2}) \# T(p,q_{3}) \# -T(p,q_{1};p,q_{3}) \# -T(p,q_{2}) \) are given~by
  \[\mathcal{K}_{s}(J) =
    \begin{cases}
      T(p,q_{2}) \# T(p,q_{3}) \# -T(p,q_{3}) \# -T(p,q_{2}),           & s = 0, \\
      T(p,q_{1}) \# -T(p,q_{1}), & s = 1, \\
      U                                                                & s \geq 2.
    \end{cases}
  \]
  Here for~$s=1$, we used that~$\mathcal{K}_{1}(J)=T(p,q_{1}) \# U \# -T(p,q_{1}) \# -U$ is~$T(p,q_{1}) \# -T(p,q_{1})$.
  In particular, observe that the formula displayed in~\eqref{eq:DecompositionTorusKnot} also holds for~$J$. 
  As we shall use in Proposition~\ref{prop:algebraic-sliceness} below, it holds for the linear combination of~\eqref{eq:LinearCombination}.
\end{example}

For later use, we note that the~$0$-level of~$K$ is the most important to us: the first homology of its~$p$-fold branched cover equals that of~$K$.
\begin{remark}
  \label{rem:BranchedCover0Level}
  Since we know that~$H_1(\Sigma_p(J_{p,q}))=H_1(\Sigma_p(T(p,q)))$ for any knot~$J$,  we deduce
  ~$$H_1(\Sigma_p(K))= H_1(\Sigma_p(\mathcal{K}_0(K)))=\bigoplus_{i=1}^k H_1(\Sigma_p(T(p,q_{i,\ell_i}))).$$
  The analogous decomposition holds for the linking form~$\lambda_p(K)$~\cite[Lemma 4]{Litherland}.
\end{remark}

The next proposition uses~$s$-levels to exhibit a criterion for the algebraic sliceness of~$K$.
\begin{proposition}\label{prop:algebraic-sliceness}
  Fix an integer \(p \geq 2\) and choose sequences of positive integers~\(Q_{i} = (q_{i,1},\ldots,q_{i,\ell_i})\) that are relatively prime to \(p\), for \(i=1,2,\ldots,k\).
  The following statements are equivalent:
  \begin{enumerate} 
  \item\label{item:algebraic-sliceness-1} the knot
    ~$K = n_{1}T(p,Q_{1}) \# \cdots \# n_{k} T(p,Q_{k})$
    is algebraically slice, 
  \item\label{item:algebraic-sliceness-2} each \(\mathcal{K}_{s}(K)\) is slice.
  \end{enumerate}
\end{proposition}
\begin{proof}
  We first assert that the polynomials~$\Delta_{\mathcal{K}_s(K)}(t^{p^s})$ and~$\Delta_{\mathcal{K}_{u}(K)}(t^{p^{u}})$ have distinct roots if~$s\neq~u$.
    For a positive integer~$m$, we set \(\xi_{m} := e^{ 2 \pi i/m }\).
    The roots of~\(\Delta_{T(p,q)}(t)\) occur at those \(\xi_{pq}^{a}\) where the integer \(1 \leq a \leq pq\) is such that neither~\(p\) nor~\(q\) divides \(a\), i.e. \(\left(\xi_{pq}^{a}\right)^{p} \neq 1\) and~\(\left( \xi_{pq}^{a} \right)^{q} \neq 1\).
    Consequently, the roots of~$\Delta_{T(p,q)}(t^{p^{s}})$ occur at \(\xi_{p^{s+1}q}^{a}\) such that \(1 \leq a \leq p^{s+1}q\) and neither \(p\) nor \(q\) divides \(a\).
  
We argue that if~$s\neq u$, then~$\Delta_{T(p,q_1)}(t^{p^s})$ and~$\Delta_{T(p,q_2)}(t^{p^{u}})$ have distinct roots.
    Assume to the contrary that they have a common root.
    This root must be of the form~$\xi_{p^{s+1}q_1}^a=\xi_{p^{u+1}q_2}^b$ where~$q_1,p$ (resp.~$q_2,p)$ do not divide~$a$ (resp. b).
    Without loss of generality, assume that~$s<u$ so that~$1=(\xi_{p^{s+1}q_1}^a)^{p^{s+1}q_1}=(\xi_{p^{u+1}q_2}^b)^{p^{s+1}q_1}=\xi_{p^{u-s}q_2}^{bq_1}$.
    This implies that~$p^{u-s}q_2$ divides~$bq_1$.
    However, by assumption,~$p$ divides neither~$q_1$ nor~$b$, yielding the desired contradiction.
  
Next, recall from the definition of the~$s$-level that
  ~$$ \mathcal{K}_{s}(K) :=n_{1} T(p,q_{1,\ell_{1}-s}) \# n_{2} T(p,q_{2,\ell_{2}-s}) \# \ldots \# n_{k} T(p,q_{k,\ell_{k}-s}).
  $$
 Thus, if~$s\neq u$, then~$\Delta_{\mathcal{K}_s(K)}(t^{p^s})$ and~$\Delta_{\mathcal{K}_{u}(K)}(t^{p^{u}})$ have distinct roots. 
This proves the assertion.

    Assume that~$K$ is algebraically slice. By the cabling formula for the Blanchfield pairing (see Example~\ref{ex:Decompo}), 
    \begin{equation}
      \label{eq:BlanchfieldLinearCombination}
      \Bl(K)(t) \cong \bigoplus_{s \geq 0}\Bl(\mathcal{K}_{s}(K))(t^{p^{s}})
    \end{equation}
    is metabolic. By the assertion and Proposition~\ref{prop:Splitting}, we deduce that each \(\Bl(\mathcal{K}_{s}(K))(t^{p^{s}})\) is metabolic. It follows that the jump function of each \(\Bl(\mathcal{K}_{s}(K))(t^{p^{s}})\) is trivial which is simply a reparametrization of the jump function of  \(\Bl(\mathcal{K}_{s}(K))(t)\) where the parameter~$t\in S^1$ is changed to~$t^{p^r}$. Hence~$\mathcal{K}_s(K)$ is a connected sum of torus knots such that the jump function of~$\sigma_{\omega}(\mathcal{K}_s)$ is trivial. Since  Litherland showed in \cite[Lemma~1]{Litherland-signature} that  the jump functions of~$\sigma_{\omega}(T(p,q))$ are linearly independent,~$\mathcal{K}_s(K)$ is slice as desired.

    Assume that each \(\mathcal{K}_{s}(K)\) is slice. As a linking form over~$\Z[t^{\pm 1}]$,~$\Bl(\mathcal{K}_{s}(K))$ is metabolic. 
    Combining this with the decomposition displayed in~\eqref{eq:BlanchfieldLinearCombination}, we deduce that~$\Bl(K)$ is metabolic, as a linking form over~$\Z[t^{\pm 1}]$.
    This is equivalent to~$K$ being algebraically slice~\cite{Kearton}, completing the proof of Proposition~\ref{prop:algebraic-sliceness}.
\end{proof}

When~$K$ is algebraically slice, we obtain a convenient description of the~$0$-level of~$K$.
\begin{corollary}
  \label{cor:0Level}
  Suppose that~$K$, \(p\), \(\ell_{i}\) and \(Q_{i}\), for \(i=1,\ldots,k\), 
  are as in Proposition~\ref{prop:algebraic-sliceness}.
  If~$K$ is algebraically slice, then~$k$ is even and, after renumbering if necessary, the~$0$-level of~$K$ is
  ~$$ \mathcal{K}_{0}(K) =
  \bigsharp_{j=1}^{k/2} m_j \left( T(p,q_{j,\ell_j})  \# -T(p,q_{j,\ell_j}) \right) 
  ~$$
\end{corollary}
\begin{proof}
  By Proposition~\ref{prop:algebraic-sliceness},~$\mathcal{K}_{0}(K)$ is a slice linear combination of torus knots.
  Since torus knots are linearly independent in the knot concordance group, the conclusion follows. 
\end{proof}

\subsection{Linear independent families of iterated torus knots}
\label{sub:free-subgr-gener}

Fix a prime power \(p\). 
The goal of this section is to prove Theorem~\ref{thm:LinIndep} whose statement we briefly recall. 
Let~$\mathcal{S}_p$ be the set of iterated torus knots~\(T(p,Q)\), where the sequences~$Q = (q_{1},q_{2},\ldots,q_{\ell})$ of \(\ell_{i}\) positive integers are coprime to~$p$ and satisfy
\begin{enumerate}
\item $q_\ell$ is a prime;
\item for \(i=1,\ldots,\ell-1\), the integer \(q_{i} \) is coprime to \(q_{\ell}\) when~$\ell >1$;
\end{enumerate}
Theorem~\ref{thm:LinIndep} states that~$\mathcal{S}_p$ is linearly independent in the topological knot concordance group~$\mathcal{C}^{\text{top}}$.

For~$i=1,\ldots, k$, we therefore choose sequences \(Q_{i} = (q_{i,1},q_{i,2},\ldots,q_{i,\ell_i})\) of positive integers where $q_{i,\ell_i}$ is prime for all~$i$, and the integer~$q_{i,j}$ is coprime to~$p$ and to~$q_{i,\ell_i}$ for all~$j$.
  We also let  \(n_{1},\ldots,n_{k} \in \Z\) be integers.
We will use metabelian Blanchfield pairings~\cite{MillerPowell,BorodzikConwayPolitarczyk} to obstruct the sliceness of the knot
\[K = n_{1}T(p,Q_{1}) \# n_{2} T(p,Q_{2}) \# \cdots \# n_{k} T(p,Q_{k}).\]
The sliceness obstruction that we will use, and which is due to Miller-Powell~\cite[Theorem~6.10]{MillerPowell}, reads as follows.
  If for every~$\Z_p$-invariant metaboliser~$G$ of~$\lambda_p(K)$, there exists a prime power order character~$\chi$ that vanishes on~$G$ and such that~$\Bl_{\alpha(p,\chi)}(K)$ is not metabolic, then~$K$ is not slice.
  Here, we use~$\alpha(p,\chi):=\alpha_K(p,\chi)$ to denote the metabelian representation that was described in Subsection~\ref{sub:MetabRep}.

\begin{remark}
The \emph{metabelian Blanchfield pairing} is a linking form
  $$ \Bl_{\alpha(p,\chi)}(K) \colon H_1(M_K;\LC_{\alpha(p,\chi)}^p) \times H_1(M_K;\LC_{\alpha(p,\chi)}^p)  \to \C(t)/\LC, $$
  where $H_1(M_K;\LC_{\alpha(p,\chi)}^p) $ denotes the homology of the $0$-framed surgery of $K$ twisted by $\alpha(p,\chi)$.
  The precise definition of $\Bl_{\alpha(p,\chi)}(K)$ is not needed in this paper (the interested reader can nonetheless find it in~\cite{MillerPowell} and~\cite{BorodzikConwayPolitarczyk}).
  All we need is the behavior of $\Bl_{\alpha(p,\chi)}(K)$ under satellite operations, and this will be recalled as the argument proceeds.
\end{remark}

The strategy behind the proof of Theorem~\ref{thm:LinIndep} is as follows.
  \begin{enumerate}
  \item Firstly, we study the characters on~$H_1(\Sigma_p(K))$.
  \item Secondly, we study the consequences of~$\Bl_{\alpha(p,\chi)}(K)$ being metabolic.
    This will impose substantial restrictions on~$\chi$.
  \item Thirdly, we build characters that violate these restrictions.
  \item Finally, we combine these first three steps to conclude the proof.
  \end{enumerate}
  The reader that wishes to see how these steps combine might consider starting with a glance at the end of the argument, after the conclusion of the proof of Lemma~\ref{lemma:constructing-characters}; see Subsection~\ref{subsub:Conclusion}.

\subsubsection{Characters on~$H_1(\Sigma_p(K))$.}
\label{subsub:Charac}
Assume that~$K$ is slice.
The first step is to study the possible characters on the~$p$-fold branched cover of~$K$.
Since~$K$ is algebraically slice, Corollary~\ref{cor:0Level} implies that~$k$ is even and, after renumbering if necessary, for some prime~$r$ (which is one of the $q_{j,\ell_j}$) and some integers~$m_1,\ldots,m_{k/2}$, we can write
\[\mathcal{K}_{0}(K) = m_1 \left( T(p,r) \# -T(p,r) \right) \#  \bigsharp_{j=2}^{k/2} m_j \left( T(p,q_{j,\ell_j})  \# -T(p,q_{j,\ell_j}) \right), \]
where~$q_{i,\ell_i}= r$ if and only if~$1 \leq i \leq 2m_{1}$.  
It follows that if we set \(M_{j} = m_{1}+m_{2}+ \cdots + m_{j-1}\), for \(j=2,\ldots,k/2\), then after further possible renumbering, the knot~$K$ can be rewritten as
\begin{equation}
  \label{eq:AlgebraicallySliceForm}
  K = \bigsharp_{i=1}^{m_{1}}\left( T(p,Q_{2i-1}) \# -T(p,Q_{2i}) \right) \# \bigsharp_{j=2}^{k/2} \bigsharp_{i=1}^{m_{j}} \left( T(p,Q_{2M_{j}+2i-1}) \# -T(p,Q_{2M_{j}+2i}) \right).
\end{equation}
As Remark~\ref{rem:BranchedCover0Level} implies that ~$H_{1}(\Sigma_{p}(K)) \cong H_{1}(\Sigma_{p}(\mathcal{K}_{0}(K)))$, the description of \(\mathcal{K}_{0}(K)\), the primary decomposition, and the fact that the~$q_{i,\ell_i}$ are prime shows that
\begin{align}
  \label{eq:DecompoBranchedCover}
  H_{1}(\Sigma_{p}(K)) =H_1(\Sigma_p(T(p,r)))^{m_{1}}  
  &\oplus H_1(\Sigma_p(-T(p,r)))^{m_{1}} \\
  &\oplus
    \bigoplus_{j=2}^{k/2} \left( H_1(\Sigma_p(T(p,q_{j,\ell_j})))^{m_{j}}  \oplus H_1(\Sigma_p(-T(p,q_{j,\ell_j})))^{m_{j}} \right). \nonumber
\end{align}
The linking form~$\lambda_p(K)$ on~$\Sigma_p(K)$ decomposes analogously.

From now on,~$\theta$ denotes the trivial character.
Also, since~$H_1(\Sigma_p(T(p,r))) \cong \Z_r^{p-1}$, we write characters~$H_1(\Sigma_p(T(p,r))) \to \Z_r$ as~$\chi_{\mathbf{a}}$ where~$\mathbf{a} \in \Z_r^p$.
Since $r$ is distinct from~$q_{i,\ell_i}$ for~$i>2m_1$, the decomposition of~\eqref{eq:DecompoBranchedCover} implies that any character~$\chi \colon H_{1}(\Sigma_{p}(K)) \to \Z_{r}$ must be of the form
\begin{equation}
  \label{eq:Charac}
  \chi= \bigoplus_{i=1}^{m_1} \left( \chi_{\mathbf{a}^i} \oplus \chi_{\mathbf{b}^i} \right) \oplus \bigoplus_{j=2}^{k/2}\bigoplus_{i=1}^{m_j} \theta \oplus \theta,
\end{equation}
where~$\lbrace \mathbf{a}^j \rbrace_{j=1}^{m_1}$ and~$\lbrace \mathbf{b}^j \rbrace_{j=1}^{m_1}$ are sequences of~$p$ elements in~$\Z_r$.
\begin{remark}
  \label{rem:PrimaryDecompositionMetaboliser}
  Recall that the Miller-Powell obstruction requires that for every~$\Z_p$-invariant metaboliser~$G$ of~$\lambda_p(K)$, we construct a prime power order character~$\chi$ that vanishes on~$G$ and such that~$\Bl_{\alpha(p,\chi)}(K)$ is not metabolic.
  The primary decomposition implies that every such metabolizer decomposes as a direct sum of metabolisers of the summands in~\eqref{eq:DecompoBranchedCover}.

  Consequently, thanks to the form of the character in~\eqref{eq:Charac},
  it is sufficient to prove the following result: for every~$\Z_p$-invariant metaboliser~$L$ of~$\lambda_p(T(p,r))^{m_1} \oplus -\lambda_p(T(p,r))^{m_1}$, there is a prime power order character~$\bigoplus_{i=1}^{m_1} \left( \chi_{\mathbf{a}^i} \oplus \chi_{\mathbf{b}^i} \right)$ that vanishes on~$L$ and such that~$\Bl_{\alpha(p,\chi)}(K)$ is not metabolic, with $\chi$ as in~\eqref{eq:Charac}.
\end{remark}

\subsubsection{The metabelian Blanchfield pairing of~$K$.}
\label{subsub:Blanchfield}
We now study the metabelian Blanchfield pairing of~$K$.
  We first use satellite formulas to decompose it, and we then study the implications of it being metabolic.
We use~$\alpha(p,\chi):=\alpha_K(p,\chi)$ to denote the metabelian representation that was described in Subsection~\ref{sub:MetabRep}.
The behavior of metabelian Blanchfield pairings under connected sums~\cite[Corollary~8.21]{BorodzikConwayPolitarczyk} implies that~$  \Bl_{\alpha(p,\chi)}(K)$ is Witt equivalent to the following linking form:
\begin{align}
  \label{eq:ApplySatelliteFormula}
  \Bl_{\alpha(p,\chi)}(K) \sim \bigoplus_{i=1}^{m_1} & \left(\Bl_{\alpha(p,\chi_{\mathbf{a}^i})}(T(p,Q_{2i-1})) \oplus -  \Bl_{\alpha(p,\chi_{\mathbf{b}^i})}(T(p,Q_{2i})) \right) \\
                                                     & \oplus \bigoplus_{j=2}^{k/2 } \bigoplus_{i=1}^{m_j} \left( \Bl_{\alpha(p,\theta)}(T(p,Q_{2 M_{j} + 2i-1}))  \oplus -  \Bl_{\alpha(p,\theta)}(T(p,Q_{2 M_{j} + 2i}))  \right). \nonumber
\end{align}
For a sequence~$S=(q_1,\ldots,q_k)$, we use~$T(p,\widehat{S})$
to denote the iterated torus knot~$T(p,q_1;\ldots;p,q_{k-1})$.
Next, we apply the satellite formula for the metabelian Blanchfield pairing~\cite[Theorem 8.19]{BorodzikConwayPolitarczyk} to both expressions in~\eqref{eq:ApplySatelliteFormula}.
As we are working with~$p$-fold covers, and the sequences~$Q_{2i-1}$ and~$Q_{2i}$ (resp.~$Q_{2 M_{j} + 2i-1}$ and~$Q_{2 M_{j} + 2i}$) both have~$r$ (resp.~$q_{j,\ell_j}$) as the prime in last position, we claim
\begin{align}
  \label{eq:TwistedSatelliteApplication}
  \Bl_{\alpha(p,\chi)}(K)
  & \sim \bigoplus_{i=1}^{m_1}  \left(\Bl_{\alpha(p,\chi_{\mathbf{a}^i})}(T(p,r)) \oplus -  \Bl_{\alpha(p,\chi_{\mathbf{b}^i})}(T(p,r)) \right)  \nonumber \\
  &\oplus \bigoplus_{i=1}^{m_1} \bigoplus_{u=1}^p  
    \left(\Bl(T(p,\widehat{Q}_{2i-1}))(\xi_r^{\mathbf{a}_u^i}t) \oplus -  \Bl(T(p,\widehat{Q}_{2i}))(\xi_r^{\mathbf{b}_u^i}t)  \right)  \\  
  &\oplus \bigoplus_{j=2}^{k/2 } \bigoplus_{i=1}^{m_j} \left( \Bl_{\alpha(p,\theta)}(T(p,q_{j,\ell_j})))  \oplus -  \Bl_{\alpha(p,\theta)}(T(p,q_{j,\ell_j}))  \right)  \nonumber  \\
  &\oplus \bigoplus_{j=2}^{k/2 } \bigoplus_{i=1}^{m_j} \bigoplus_{u=1}^p   \left( \Bl(T(p,\widehat{Q}_{2 M_{j} + 2i-1}))(t)  \oplus -  \Bl(T(p,\widehat{Q}_{2 M_{j} + 2i}))(t) \right). \nonumber
\end{align}
The satellite formula of~\cite[Theorem~8.19]{BorodzikConwayPolitarczyk} involves the expression~$\operatorname{Bl}(K) (\xi_{q_1}^{\chi(t_Q^{i-1}q_Q(\mu_Q^{-w}\eta))} t)$, where~$\mu_Q$ denotes the meridian of the satellite knot~$Q=P_\eta(K)$ with pattern~$P$, companion~$K$ and infection curve~$\eta$; furthermore,~$q_Q\colon \pi_1(M_Q) \to H_1(\Sigma_p(Q))$ denotes the map described in~\eqref{eq:qK}.
Recalling the notations of Section~\ref{sec:branch-covers-torus}, we see that in our case,~$\eta$ coincides with the curve~$c_2$, and~$\mu_Q=\mu_{T(p,q)}$.
Thus, as explained in~\eqref{eq:PracticalForCharacter} for~$\chi=\chi_{\mathbf{a}}$, we deduce that~$\chi(t_Q^{u-1}q_Q(\mu_Q^{-w}\eta))=\mathbf{a}_u$, and this explains the second summand of~\eqref{eq:TwistedSatelliteApplication}.
The decomposition in~\eqref{eq:TwistedSatelliteApplication} is now justified, concluding the claim.

Next, we wish to apply the cabling formula~$\Bl(J_{p,q})(t)=\Bl(T(p,q))(t) \oplus \Bl(J)(t^p)$ for the classical Blanchfield pairing.
To make notations more manageable however, for \(s \geq 1\), coprime integers~$p,q$ and~$\mathbf{a} \in \Z_r^{p}$, we consider the linking form
\[\Lambda(p,q,\chi_\mathbf{a},s) := \bigoplus_{u=0}^{p-1}\Bl(T(p,q))(\xi_{r}^{p^{s-1}\mathbf{a}_{u}}t^{p^{s-1}}).\]
If the character~$\chi_{\mathbf{a}}$ is trivial, then we write~$\Lambda(p,q,s)$ instead of~$\Lambda(p,q,\theta,s)$.
These pairings appear as summands of the Blanchfield pairing of a cable. Indeed,
using these notations and the aforementioned untwisted cabling formula, we deduce from~\eqref{eq:TwistedSatelliteApplication} that 
\begin{align}
  \Bl_{\alpha(p,\chi)}(K) &\sim 
                            \bigoplus_{i=1}^{m_1} \left(\Bl_{\alpha(p,\chi_{\mathbf{a}^{i}})}(T(p,r)) \oplus -\Bl_{\alpha(p,\chi_{\mathbf{b}^{i}})}(T(p,r)) \right) \label{eq:term-one}\\
                          &\oplus \bigoplus_{j=2}^{k/2}  \bigoplus_{i=1}^{m_j} \left(\Bl_{\alpha(p,\theta)}(T(p,q_{j,\ell_{j}})) \oplus -Bl_{\alpha(p,\theta)}(T(p,q_{j,\ell_{j}})) \right) \label{eq:term-two}\\
                          &\oplus \bigoplus_{i=1}^{m_1} \bigoplus_{s \geq 1} \left( \Lambda(p,q_{2i-1,\ell_{2i-1}-s},\chi_{\mathbf{a}^{i}},s)  \oplus - \Lambda(p,q_{2i,\ell_{2i}-s},\chi_{\mathbf{b}^{i}},s) \right) \label{eq:term-three}\\
                          &\oplus \bigoplus_{j=1}^{k/2} \bigoplus_{i=2}^{m_j} \bigoplus_{s \geq 1} \left( \Lambda(p,q_{2M_j+2i-1,\ell_{2M_j+2i-1}-s},s) \oplus - \Lambda(p,q_{2M_j+2i,\ell_{2M_j+2i}-s},s)\right). \label{eq:term-four}
\end{align}
To simplify the notation, we respectively use \(B_1^\chi,B_2,B_3^\chi,B_4\) to denote~\eqref{eq:term-one},~\eqref{eq:term-two},~\eqref{eq:term-three} and~\eqref{eq:term-four}.
                          
Now that we have decomposed~$\Bl_{\alpha(p,\chi)}(K)$, we study the consequences of it being metabolic.

\begin{claim}
  \label{claim:T1T3PlusT4Metabolic}
  If \(\Bl_{\alpha(p,\chi)}(K)\) is metabolic, then \(B_1^\chi\) and \(B_3^\chi \oplus B_{4}\) are metabolic.
\end{claim}
\begin{proof}
  As~$\Bl_{\alpha(p,\chi)}(K)$ and~$B_2$ are metabolic,~$B_1^\chi \oplus (B_3^\chi \oplus B_4)$ is metabolic. By Proposition ~\ref{prop:Splitting}, it suffices to prove that the orders of  \(B_1^\chi\) and \(B_3^\chi \oplus B_{4}\) have distinct roots: the roots of the twisted polynomial occur at prime powers of unity (by Proposition~\ref{prop:TwistedAlexanderPolynomial}), while this is never the case for the classical Alexander polynomial~\cite[proof of Proposition~3.3, item~(3)]{FriedlEta}.\footnote{Here is a topological proof of this fact: for a knot~$K$ and an integer~$q$, the order of~$H_1(\Sigma_{q}(K))$ is~$\prod_{a=1}^{q-1}\Delta_K(\xi_{q}^a)$~\cite[Corollary~9.8]{LickorishIntroduction}; since~$q$ is a prime power,~$H_1(\Sigma_{q}(K))$ is a finite group, and thus none of the~$\Delta_K(\xi_{q}^a)$ can vanish.}
This proves of Claim~\ref{claim:T1T3PlusT4Metabolic}.
\end{proof}

In order to study the consequences of \(B_3^\chi \oplus B_{4}\) being metabolic, for \(s \geq 1\), we set
\begin{align*}
  B_3^\chi(s) &:= \bigoplus_{i=1}^{m_1} \left( \Lambda(p,q_{2i-1,\ell_{2i-1}-s},\chi_{\mathbf{a}^{i}},s)  \oplus - \Lambda(p,q_{2i,\ell_{2i}-s},\chi_{\mathbf{b}^{i}},s) \right), \\
  B_{4}(s) &:= \bigoplus_{j=1}^{k/2} \bigoplus_{i=2}^{m_j} \left( \Lambda(p,q_{2M_j+2i-1,\ell_{2M_j+2i-1}-s},s) \oplus - \Lambda(p,q_{2M_j+2i,\ell_{2M_j+2i}-s},s)\right).
\end{align*}
Using these forms, we derive a further consequence of~$\Bl_{\alpha(p,\chi)}(K)$ being metabolic.

\begin{claim}
  \label{claim:Ti(s)Metabolic}
  If~$B_3^\chi \oplus B_4$ is metabolic, then~$B_3^\chi(s) \oplus B_{4}(s)$ is metabolic for each~$s$.
\end{claim}
\begin{proof}
  By definition, we have the decompositions~$B_3^\chi = \bigoplus_{s \geq 1} B_3^\chi(s)$ and~$B_{4} = \bigoplus_{s \geq 1} B_{4}(s)$.
  For~$u \neq v$, the order of~$B_3^\chi(u) \oplus B_4(u)$ and the order of~$B_3^\chi(v) \oplus B_4(v)$ have distinct roots. By Proposition~\ref{prop:Splitting}, Claim~\ref{claim:Ti(s)Metabolic} follows.
\end{proof}

Consequently, it is sufficient to study the linking forms \(B_3^\chi(s) \oplus B_{4}(s)\), for a fixed \(s \geq 1\).
To further decompose~$B_3^\chi(s) \oplus B_4(s)$, we want to group these linking forms according to the torus knots that appear.
We also need to be attentive to the fact that the torus knot~$T(p,q_{i,\ell_i-s})$ is trivial when~$i \leq \ell_i$.
As a consequence, for \(s \geq 1\), we consider the sets 
\begin{align}
  \label{eq:I_j(q,s)}
  \mathcal{I}_{1}(q,s) &:= \{1 \leq i \leq m_{1} \ \big\vert \ \ell_{2i-1} > s, \quad q_{2i-1,\ell_{2i-1}-s}=q\},  \nonumber \\
  \mathcal{I}_{2}(q,s) &:= \{1 \leq i \leq m_{1} \ \big\vert \  \ell_{2i} > s, \quad q_{2i,\ell_{2i}-s}=q\}, \\
  \mathcal{I}_{3}(q,s) &:= \bigcup_{j=2}^{k/2} \{1 \leq i \leq m_{j} \ \big\vert \  \ell_{2M_{j} + 2i-1} > s, \quad q_{2M_{j}+2i-1,\ell_{2M_{j}+2i-1}-s}=q\}, \nonumber \\
  \mathcal{I}_{4}(q,s) &:= \bigcup_{j=2}^{k/2} \{1 \leq i \leq m_{j} \ \big\vert \  \ell_{2M_{j} + 2i} > s, \quad q_{2M_{j} + 2i,\ell_{2M_{j}+2i}-s}=q\}. \nonumber
\end{align}

Note that for some~$q$, the set~$\mathcal{I}_{i}(q,s)$ may well be empty.
However, from now on, we will implicitly assume that we only consider~$q$ for which this is not the case.
In order to study the consequences of~$B_3^\chi(s) \oplus B_4(s)$ being metabolic, we set
\begin{align*}
  B_3^\chi(q,s) &:= \bigoplus_{k \in \mathcal{I}_{1}(q,s)} \Lambda(p,q,\chi_{\mathbf{a}^{k}},s) \oplus -\bigoplus_{k \in \mathcal{I}_{2}(q,s)} \Lambda(p,q,\chi_{\mathbf{b}^{k}},s), \\
  B_4(q,s) &:= \bigoplus_{k \in \mathcal{I}_{3}(q,s)} \Lambda(p,q,s) \oplus -\bigoplus_{k \in \mathcal{I}_{4}(q,s)} \Lambda(p,q,s).
\end{align*}

Note that~$B_4(q,s)$ is not automatically metabolic as the cardinality of~$  \mathcal{I}_{3}(q,s)$ need not agree with that of~$  \mathcal{I}_{4}(q,s)$.
Observe however that if \(K\) is algebraically slice,
  Proposition~\ref{prop:algebraic-sliceness} implies that
\begin{equation}
  \label{eq:1234}
  \# \mathcal{I}_{1}(q,s) - \# \mathcal{I}_{2}(q,s) + \# \mathcal{I}_{3}(q,s) - \# \mathcal{I}_{4}(q,s) = 0.
\end{equation}
Indeed, note that the sets \(\mathcal{I}_{i}(q,s)\)
  record where \(T(p,q)\) appears in the \(s\)-level of \(K\).
Using the~$B_i(q,s)$, we now derive a further consequence of~$\Bl_{\alpha(p,\chi)}(K)$ being metabolic.

\begin{claim}
  \label{claim:Ti(q,s)Metabolic}
  If~$B_3^\chi(s) \oplus B_4(s)$ is metabolic, then~$B_3^\chi(q,s) \oplus B_4(q,s)$ is metabolic for each~$q$.
\end{claim}
\begin{proof}
  We have the decompositions~$B_3^\chi(s) = \bigoplus_{q \geq 1} B_3^\chi(q,s)$ and~$B_{4}(s) = \bigoplus_{q \geq 1} B_{4}(q,s)$.  Since all the~$q_i$ are positive, for~$u \neq v$, the order of~$B_3^\chi(u,s) \oplus B_4(u,s)$ and the order of~$B_3^\chi(v,s) \oplus B_4(v,s)$ have distinct roots. By Proposition~\ref{prop:Splitting}, Claim~\ref{claim:Ti(q,s)Metabolic} follows.
\end{proof}

Summarising the content of these claims, we have shown that if the metabelian Blanchfield pairing~$\Bl_{\alpha(p,\chi)}(K)$ is metabolic, then the linking forms~$B_3^\chi(q,s) \oplus B_4(q,s)$ are metabolic for all~$q,s$.
  This concludes the second part of the proof.

\subsubsection{Building the characters that vanish on metabolisers}
\label{subsub:BuildingCharacters}
  The third part consists in showing that for every~$\Z_p$-invariant metaboliser $L$ of~$\lambda_p(T(p,r))^{m_1} \oplus -\lambda_p(T(p,r))^{m_1}$ there are characters $\chi_{\mathbf{a}}=\bigoplus_{i=1}^{m_1} \chi_{\mathbf{a}^i}$ and $\chi_{\mathbf{b}}=\bigoplus_{i=1}^{m_1} \chi_{\mathbf{b}^i}$ such that $\chi_{\mathbf{a}} \oplus \chi_{\mathbf{b}}$ vanishes on $L$, but for which the linking forms~$B_3^\chi(q,s) \oplus B_4(q,s)$ are not all metabolic, where $\chi=\chi_{\mathbf{a}} \oplus \chi_{\mathbf{b}} \oplus \theta$ is as in~\eqref{eq:Charac}.

The next proposition describes characters for which~$B_3^\chi(q,s) \oplus B_4(q,s)$ is not metabolic.

\begin{proposition}
  \label{prop:NotMetabolic}
  Let \(q,s>0\) be positive integers with~$q$ coprime to~$p$. If a character \linebreak \(\bigoplus_{i=1}^{m_{1}}
  \chi_{\mathbf{a}^{i}} \oplus \chi_{\mathbf{b}^{i}}\) satisfies one of the following
  conditions:
  \begin{enumerate}
  \item \(\chi_{\mathbf{b}^{k}} = \theta\) for every \(k \in I_{2}(q,s)\) and
    \(\chi_{\mathbf{a}^{k_{0}}} \neq \theta\) for some \(k_{0} \in
    I_{1}(q,s)\), or,
  \item \(\chi_{\mathbf{a}^{k}} = \theta\) for every \(k \in I_{1}(q,s)\) and
    \(\chi_{\mathbf{b}^{k_{0}}} \neq \theta\) for some \(k_{0} \in
    I_{2}(q,s)\),
  \end{enumerate}
  then the linking form \(B_{3}^{\chi}(q,s) \oplus B_{4}(q,s)\) is not metabolic.
\end{proposition}
\begin{proof}
  We will only consider case~(1).
  In order to give the proof in case~(2) just exchange the roles of
  \(\chi_{\mathbf{a}}\) and \(\chi_{\mathbf{b}}\).
  Assume that \(\chi_{\mathbf{b}^{k}} = \theta\) for every \(k \in \mathcal{I}_{2}(q,s)\) and
  \(\chi_{\mathbf{a}^{k_{0}}} \neq \theta\) for some \(k_{0} \in
  \mathcal{I}_{1}(q,s)\).
  Since \(K\) is algebraically slice, recall from~\eqref{eq:1234} that
  \[\#\mathcal{I}_{1}(q,s) - \#\mathcal{I}_{2}(q,s) + \#\mathcal{I}_{3}(q,s) - \#\mathcal{I}_{4}(q,s) = 0,\]
  We thus define \(N := \#\mathcal{I}_{1}(q,s) = \#\mathcal{I}_{2}(q,s) - \#\mathcal{I}_{3}(q,s)+\#\mathcal{I}_{4}(q,s)\) leading to the Witt equivalence
  \begin{equation}
    \label{eq:NotMetabolic}
    B_{3}^{\chi}(q,s) \oplus B_{4}(q,s) \sim \bigoplus_{k \in I_{1}(q,s)}
    \Lambda(p,q,\chi_{\mathbf{a}^{i}},s) \oplus- \bigoplus_{i=1}^{p\cdot N}
    \Bl(T(p,q))(t^{p^{s-1}}).
  \end{equation}
We assert that the orders of the modules underlying the summands of  the right hand side of~\eqref{eq:NotMetabolic} have distinct roots.
    First, note that~$r$ is coprime to~$q$: as~$k \in \mathcal{I}_1(q,s)$, we know that~$q \in Q_i$ for some~$i<~2m_1$, and since~$Q_i=(q_{i,1},q_{i,2},\ldots,q_{i,{\ell_i-1}},r)$ for~$i<2m_1$, this follows from the assumption of Theorem~\ref{thm:LinIndep}. 
    It is known that~$\Delta_{T(p,q)}(\xi^{a_1}_rt)$ and~$\Delta_{T(p,q)}(\xi^{a_2}_rt)$ have distinct roots whenever~$a_1 \neq~a_2$ and $r$ and $q$ are coprime~\cite[Theorem~7.1]{HeddenKirkLivingston}.
    This establishes the assertion.

  Thanks to the assertion, we may apply Proposition~\ref{prop:NotMetabolic}.
    Indeed, the fact that \(\chi_{\mathbf{a}^{k_{0}}} \neq \theta\) and
  Proposition~4.3 now guarantees that the linking form on the right-hand side of~\eqref{eq:NotMetabolic} is not metabolic.
  This concludes the proof of Proposition~\ref{prop:NotMetabolic}.
\end{proof}

Before constructing the required characters, we introduce some terminology.
We say that the knot~\(K\) is \emph{simplified}, if there are no indices~\(k_{1} \in \mathcal{I}_1(q,s) \) and~\(k_{2} \in \mathcal{I}_2(q,s) \) such that \(Q_{2k_{1}-1} = Q_{2k_{2}}\).
If~\(K\) is not simplified, then it contains a slice connected summand \(T(p,Q_{2k_{1}-1}) \# -T(p,Q_{2k_{1}-1})\).

\begin{lemma}
  \label{lemma:constructing-characters}
  Let \(p\) be a prime power.
  If the knot \(K\) is simplified, then for any \(\Z_{p}\)-invariant
  metabolizer \(L \subset H_1(\Sigma_p(T(p,r)))^{m_{1}} \oplus H_1(\Sigma_p(T(p,r)))^{m_{1}}  \) there exist \(q,s\) and a character \(\ \chi_{\mathbf{a}} \oplus \chi_{\mathbf{b}}=\bigoplus_{i=1}^{m_1} \chi_{\mathbf{a}^i} \oplus \chi_{\mathbf{b}^i} \) vanishing on \(L\) such that one
  of the following conditions is satisfied:
  \begin{enumerate}
  \item \(\chi_{\mathbf{b}^{k}} = \theta\) for every \(k \in \mathcal{I}_{2}(q,s)\) and
    \(\chi_{\mathbf{a}^{k_{0}}} \neq \theta\) for some \(k_{0} \in
    \mathcal{I}_{1}(q,s)\), or,
  \item \(\chi_{\mathbf{a}^{k}} = \theta\) for every \(k \in \mathcal{I}_{1}(q,s)\) and
    \(\chi_{\mathbf{b}^{k_{0}}} \neq \theta\) for some \(k_{0} \in
    \mathcal{I}_{2}(q,s)\).
  \end{enumerate}
\end{lemma}
\begin{proof}
  Fix a metabolizer \(L \subset H_1(\Sigma_p(T(p,r)))^{m_{1}} \oplus H_1(\Sigma_p(T(p,r)))^{m_{1}}  \) of~$\lambda_p(T(p,r))^{m_1} \oplus -\lambda_p(T(p,r))^{m_1}$.
  For~$i=1,2$ consider the projection 
  \(\pr_{i} \colon H_1(\Sigma_p(T(p,r)))^{m_{1}}  \oplus
  H_1(\Sigma_p(T(p,r)))^{m_{1}}  \to H_1(\Sigma_p(T(p,r)))^{m_{1}}  \) onto the \(i\)-th factor.
  The proof is divided into three separate cases.
  
  \textbf{Case 1:} \textit{\(\pr_{1}(L)\) is a proper subspace of
    \( H_1(\Sigma_p(T(p,r)))^{m_{1}} \).}
  In this case, we can define the characters~\(\chi_{\mathbf{a}}\) and
  \(\chi_{\mathbf{b}}\) as follows:  \(\chi_{\mathbf{b}} = \theta\) and  
  \[\chi_{\mathbf{a}} \colon H_1(\Sigma_p(T(p,r)))^{m_{1}} \to H_1(\Sigma_p(T(p,r)))^{m_{1}}/
    \pr_{1}(L) \xrightarrow{\text{nontrivial charater}} \Z_{r}.\]
  It is not difficult to see that \(\chi_{\mathbf{a}}\) and
  \(\chi_{\mathbf{b}}\) satisfy~(1) and are such that~$\chi_{\mathbf{a}} \oplus \chi_{\mathbf{b}}$ vanishes on~$L$.

  \textbf{Case 2:} \textit{\( \pr_{2}(L)\) is a proper subspace of
    \( H_1(\Sigma_p(T(p,r)))^{m_{1}}\).}
  In this case, we exchange the roles of~\(\chi_{\mathbf{a}}\) and~\(\chi_{\mathbf{b}}\) and repeat the argument from the first
  case.
  This way, we obtain characters \(\chi_{\mathbf{a}}\) and~\(\chi_{\mathbf{b}}\) that satisfy~(2) and are such that~$\chi_{\mathbf{a}} \oplus \chi_{\mathbf{b}}$ vanishes on~$L$.

  \textbf{Case 3:} \( \pr_{1}(L) =H_1(\Sigma_p(T(p,r)))^{m_{1}} \) and \( \pr_{2}(L) = H_1(\Sigma_p(T(p,r)))^{m_{1}}\).
  We wish to apply Proposition~\ref{prop:DirectSummetaboliser} in order to prove that~$L$ is a graph.
  We verify the hypothesis of this proposition.
  Using the assumption of Case 3 and the definition of the projections, we have
  \[0 = \ker(\pr_{1}|_{L}) = L \cap (0 \oplus H_1(\Sigma_p(T(p,r)))^{m_{1}}), \quad
    0 = \ker(\pr_{2}|_{L}) = L \cap (H_1(\Sigma_p(T(p,r)))^{m_{1}} \oplus 0).\]
  Consequently, by Proposition~\ref{prop:DirectSummetaboliser}, \(L\) is the graph of an anti-isometry
  \[g \colon (H_1(\Sigma_p(T(p,r)))^{m_{1}},\lambda_p(T(p,r))^{m_{1}}) \to
    (H_1(\Sigma_p(T(p,r)))^{m_{1}},\lambda_p(T(p,r))^{m_{1}}).\]
    For each \(q,s\) and~$j=1,2$, consider the following subsets of~$H_1(\Sigma_p(T(p,r))^{m_{1}}$
    \begin{align*}
      S_{\mathcal{I}_{j}(q,s)} 
      &= \{(v_{1},v_{2},\ldots,v_{m_{1}}) \in (H_1(\Sigma_p(T(p,r)))^{m_{1}} \colon
        v_{i} = 0, \text{ for } i \not\in \mathcal{I}_{j}(q,s)\} \\
      &= \bigoplus_{k \in    \mathcal{I}_{j}(q,s)}  H_1(\Sigma_p(T(p,Q_k)),
    \end{align*}
    where~$\mathcal{I}_j(q,s)$ is defined in \eqref{eq:I_j(q,s)}.
  
  Next, we use these sets and the anti-isometry~$g$ to describe a sufficient criterion to obtain the characters \(\chi_{\mathbf{a}},\chi_{\mathbf{b}}\) required by the statement of Lemma~\ref{lemma:constructing-characters}.
  \begin{claim}\label{claim:claim}
    If there exist \(q,s\) such that \(g(S_{\mathcal{I}_{1}(q,s)}) \neq S_{\mathcal{I}_{2}(q,s)}\), then there are 
    characters \(\chi_{\mathbf{a}},\chi_{\mathbf{b}}\) satisfying either~(1) or~(2) and such that~$\chi_{\mathbf{a}} \oplus \chi_{\mathbf{b}}$ vanishes on~$L$.
  \end{claim}
  \begin{proof}
    If \(g(S_{\mathcal{I}_{1}(q,s)}) \setminus S_{\mathcal{I}_{2}(q,s)} \neq \emptyset\), then
    choose \(v \in S_{\mathcal{I}_1(q,s)}\) such that \(g(v) \not \in
    S_{\mathcal{I}_2(q,s)}\).
    Since~$r$ is a prime,~$H_1(\Sigma_p(T(p,r)))^{m_{1}}$ is an~$\F_r$-vector space and so we obtain a direct sum decomposition~\( H_1(\Sigma_p(T(p,r)))^{m_{1}}= \langle v \rangle \oplus W\) for some~$\F_r$-vector-space~$W$.
    We can then define the characters~as
    \[\chi_{\mathbf{a}}(v) = 1, \quad \chi_{\mathbf{a}}|_{W} = \theta, \quad
      \chi_{\mathbf{b}}(x) =- \chi_{\mathbf{a}}(g^{-1}(x)).\]
    Such choices of \(\chi_{\mathbf{a}}\) and \(\chi_{\mathbf{b}}\) satisfy condition~(1). We verify that~$\chi_{\mathbf{a}} \oplus \chi_{\mathbf{b}}$ vanishes on~$L$; where we recall that $L$ is the graph of $g$.
      For an element~$(h,g(h)) \in L$ of this graph, one \linebreak has~$(\chi_{\mathbf{a}} \oplus \chi_{\mathbf{b}})(h,g(h))=\chi_{\mathbf{a}}(h)-\chi_{\mathbf{a}}(g^{-1}(g(h)))=0$.
      This concludes the proof in this~case.

If on the other hand, we assume that \(S_{\mathcal{I}_2(q,s)} \setminus g(S_{\mathcal{I}_1(q,s)}) \neq
    \emptyset\), and the argument is nearly identical.
    Choose \(v \in S_{\mathcal{I}_2(q,s)} \setminus g(S_{\mathcal{I}_1(q,s)})\) and write once
    more \( H_1(\Sigma_p(T(p,r)))^{m_{1}}= \langle v \rangle \oplus W\) and define the required characters as
    \[\chi_{\mathbf{b}}(v) = 1, \quad \chi_{\mathbf{b}}|_{W} = \theta, \quad
      \chi_{\mathbf{a}}(x) =- \chi_{\mathbf{b}}(g(x)).\]
    These choices of \(\chi_{\mathbf{a}}\) and \(\chi_{\mathbf{b}}\) satisfy    condition~(2) and  \(\chi_{\mathbf{a}} \oplus \chi_{\mathbf{b}}\) vanishes on $L$.
    This concludes the proof of the claim.
  \end{proof}
    By Claim~\ref{claim:claim}, to prove Lemma~\ref{lemma:constructing-characters}, it is enough to show that there always exist \(q,s\) such that~\(g(S_{\mathcal{I}_{1}(q,s)}) \neq S_{\mathcal{I}_{2}(q,s)}\). Assume by way of contradiction that we
    have~\(g(S_{\mathcal{I}_1(q,s)}) = S_{\mathcal{I}_2(q,s)}\) for all \(q,s\). We will show in Claim~\ref{claim:NonEmpty} below that this assumption implies that~$K$ is not simplified. This  is a contradiction since we assumed that~$K$ is simplified. This proves Lemma~\ref{lemma:constructing-characters} modulo Claim~\ref{claim:NonEmpty}.
\end{proof}

\begin{claim}
  \label{claim:NonEmpty}
  If \(g(S_{\mathcal{I}_{1}(q,s)})= S_{\mathcal{I}_{2}(q,s)}\) for all \(q,s\), then~$K$ is not simplified.
\end{claim}
\begin{proof} We will observe that under the assumption of the claim, $K$ contains a summand of the form~\(T(p,Q_{2k_{0}-1}) \# -T(p,Q_{2k_{0}-1})\) for some integer~$k_0$. To be precise, choose 
    \(1 \leq k_{0} \leq m_{1}\) such that the length~$\ell_{2k_0-1}$ of the sequence of~$Q_{2k_0-1}$ is maximal among all the~$\ell_{2k-1}$ for~$k=1,\ldots,m_1$, and define \footnote{Note that without the maximality assumption on~$\ell_{2k_0-1}$, we would have had to replace the condition~$Q_{k_0} = Q_k$ by~$Q_{k_0} \subset Q_k$.}
    \begin{align*} 
      X(k_{0}) 
      &=\lbrace 1 \leq k \leq m_1 \ | \ Q_{2k_0-1} = Q_{2k-1} \rbrace 
        = \bigcap_{s=1}^{\ell_{2k_{0}-1}} \mathcal{I}_{1} (q_{2k_{0}-1,\ell_{2k_{0}-1}-s},s), \\
      Y(k_{0}) 
      &=\lbrace 1 \leq k \leq m_1 \ | \ Q_{2k_0-1} = Q_{2k} \rbrace
        =\bigcap_{s=1}^{\ell_{2k_{0}-1}} \mathcal{I}_{2} (q_{2k_{0}-1,\ell_{2k_{0}-1}-s},s).
    \end{align*}
    We will need the following properties of these sets:
    \begin{enumerate}[label=(\alph*)]
    \item\label{item:keyproperty-1} since \(k_{0} \in
      X(k_{0})\),  \(X(k_{0})\) is nonempty;
    \item\label{item:keyproperty-2} if \(j \in X(k_{0})\), then~\(T(p,Q_{2j-1}) = T(p,Q_{2k_{0}-1})\);
    \item\label{item:keyproperty-3} if~\(j \in Y(k_{0})\), then~\(T(p,Q_{2j}) = T(p,Q_{2k_{0}-1})\).
    \end{enumerate}
    It is enough to show that \(Y(k_{0}) \neq \emptyset\). By~\ref{item:keyproperty-1}--\ref{item:keyproperty-3},  this would imply that  \(K\) is not simplified since~$K$ contains a summand of the form \(T(p,Q_{2k_{0}-1}) \# -T(p,Q_{2k_{0}-1})\).

To show that \(Y(k_{0}) \neq \emptyset\),  consider the following subspaces of \(H_1(\Sigma_p(T(p,r)))^{m_{1}}\):
    \begin{align*}
      S_{X(k_{0})} 
      &:= \{(v_1,v_2,\ldots,v_{m_1}) \in H_1(\Sigma_p(T(p,r)))^{m_{1}} \colon v_i = 0
        \text{ for } i \not\in X(k_0)\}, \\
      &= \bigoplus_{k \in X(k_0)} H_1(\Sigma_p(T(p,Q_{2k-1}))),       \\
      S_{Y(k_{0})}
      &:= \{(v_1,v_2,\ldots,v_{m_1} \in H_1(\Sigma_p(T(p,r)))^{m_{1}} \colon
        v_i = 0 \text{ for } i \not\in Y(k_0)\} \\
      &= \bigoplus_{k \in Y(k_0)} H_1(\Sigma_p(T(p,Q_{2k}))).
    \end{align*}
    The advantage of writing~$X(k_0)$ and~$Y(k_0)$ as intersections of the~$\mathcal{I}_j(q_{k_0,\ell_{k_0-s}},s)$ is that the action of \(g\) on \(S_{X(k_{0})}\) can be described as
    \[
      g(S_{X(k_{0})})
      =  \bigcap_{s \geq 1} S_{g(\mathcal{I}_{1}(q_{2k_{0}-1,\ell_{2k_{0}-1}-s},s))} 
      = \bigcap_{s \geq 1} S_{\mathcal{I}_{2}(q_{2k_{0}-1,\ell_{2k_{0}-1}-s},s)} 
      = S_{Y(k_{0})},\]
    where the second equality follows from the assumption.
    As \(g\) is an~$\F_r$-linear automorphism,~$\dim S_{X(k_{0})} = \dim S_{Y(k_{0})}$.
    Since the $\F_r$-dimension of~$H_1(\Sigma_p(T(p,r)))$ is~$p-1$, we deduce that
    \[(p-1) \# X(k_{0}) = \dim S_{X(k_{0})} = \dim S_{Y(k_{0})} = (p-1) \# Y(k_{0}).\]
    It follows that \(\# X(k_{0}) = \# Y(k_{0})\).
    Since \(X(k_{0}) \neq \emptyset\) by~\ref{item:keyproperty-1}, it follows that \(Y(k_{0}) \neq \emptyset\). As we mentioned, this implies that~$K$ is not simplified by \ref{item:keyproperty-1}--\ref{item:keyproperty-3} and Claim~\ref{claim:NonEmpty} is proved.
\end{proof}  

This concludes the third part of the proof. We can now conclude.

\subsubsection{Conclusion of the proof}
\label{subsub:Conclusion}

We can now prove Theorem~\ref{thm:LinIndep}.
\begin{proof}[Proof of Theorem~\ref{thm:LinIndep}]
  Let~$K$ be a (non-trivial) linear combination of iterated torus knots of the form~$T(p,Q_i)$ for~$i=1,\ldots, k$.
  Here, the~\(Q_{i} = (q_{i,1},q_{i,2},\ldots,q_{i,\ell_i})\) are sequences of~\(\ell_{i}\) positive integers where $q_{i,\ell_i}$ is prime for all~$i$, and the integer~$q_{i,j}$ is coprime to~$p$ and to~$q_{i,\ell_i}$ for all~$j$.
  Assume that~$K$ is slice to obtain a contradiction.
  In particular~$K$ is algebraically slice and, as we saw in~\eqref{eq:AlgebraicallySliceForm}, we can therefore assume without loss generality that it is of the form
  \begin{equation}
    \label{eq:LinearCombination}
    K =   \bigsharp_{i=1}^{m_{1}}\left( T(p,Q_{2i-1}) \# -T(p,Q_{2i}) \right) \# \bigsharp_{j=2}^{k/2} \bigsharp_{i=1}^{m_{j}} \left( T(p,Q_{2M_{j}+2i-1}) \# -T(p,Q_{2M_{j}+2i}) \right).
  \end{equation}
  Here we arranged that ~$q_{i,\ell_i}= r$ if and only if~$1 \leq i \leq 2m_{1}$. 
  Furthermore, we can assume that~$K$ is simplified by canceling terms of the form~$J \# -J$ if any such term appears in~\eqref{eq:LinearCombination}.
  We can also assume that there is an index~$i$ such that~$\ell_i>1$: otherwise~$K$ would be a linear combination of torus knots, which is impossible since the latter are linearly independent in~$\mathcal{C}^{\text{top}}$~\cite{Litherland-signature}.
  To prove that~$K$ is not slice, we saw that it is enough to show that for every~$\Z_p$-invariant metaboliser~$L$ of~$\lambda_{p}(T(p,r))^{m_1} \oplus -\lambda_{p}(T(p,r))^{m_1}$, there is a character~$\chi_{\mathbf{a}} \oplus \chi_{\mathbf{b}}=\bigoplus_{k=1}^{m_1} \left( \chi_{\mathbf{a}^k} \oplus \chi_{\mathbf{b}^k} \right)$ that vanishes on~$L$ such that~$\Bl_{\alpha(p,\chi)}(K)$ is not metabolic, where~$\chi=\chi_{\mathbf{a}} \oplus \chi_{\mathbf{b}} \oplus \bigoplus_{j=2}^{k/2} \bigoplus_{i=1}^{m_j} \theta \oplus \theta$; recall Remark~\ref{rem:PrimaryDecompositionMetaboliser}.
  We then applied satellite formulas to show that~$\Bl_{\alpha(p,\chi)}(K)$ decomposes (up to Witt equivalence) as
  \begin{align*}
    \Bl_{\alpha(p,\chi)}(K) 
    &\sim B_1^\chi \oplus B_2 \oplus B_3^\chi \oplus B_4 \\
    &=B_1^\chi \oplus B_2 \oplus \bigoplus_{q,s} B_3^\chi(q,s) \oplus \bigoplus_{q,s} B_4(q,s).
  \end{align*}
  Claim~\ref{claim:T1T3PlusT4Metabolic} shows that if~$\Bl_{\alpha(p,\chi)}(K)$ is metabolic, then~$B_1^\chi$ and~$B_3^\chi \oplus B_4$ are metabolic.
  By Claims~\ref{claim:Ti(s)Metabolic} and~\ref{claim:Ti(q,s)Metabolic}, it follows that~$B_3^\chi(q,s) \oplus B_3^\chi(q,s)$ must be metabolic \textit{for all}~$q,s$ and all characters~$\chi_{\mathbf{a}} \oplus \chi_{\mathbf{b}}$.
  On the other hand, as the knot~$K$ is simplified, Lemma~\ref{lemma:constructing-characters} implies that for any \(\Z_{p}\)-invariant
  metabolizer \(L \subset H_1(\Sigma_p(T(p,r)))^{m_{1}} \oplus H_1(\Sigma_p(T(p,r)))^{m_{1}}  \) there exist \(q,s\) and a character \(
  \chi_{\mathbf{a}} \oplus \chi_{\mathbf{b}}\) vanishing on \(L\) such that one
  of the following conditions is satisfied:
  \begin{enumerate}
  \item \(\chi_{\mathbf{b}^{k}} = \theta\) for every \(k \in \mathcal{I}_{2}(q,s)\) and
    \(\chi_{\mathbf{a}^{k_{0}}} \neq \theta\) for some \(k_{0} \in
    \mathcal{I}_{1}(q,s)\), or,
  \item \(\chi_{\mathbf{a}^{k}} = \theta\) for every \(k \in \mathcal{I}_{1}(q,s)\) and
    \(\chi_{\mathbf{b}^{k_{0}}} \neq \theta\) for some \(k_{0} \in
    \mathcal{I}_{2}(q,s)\).
  \end{enumerate}  
  Applying Proposition~\ref{prop:NotMetabolic}, we deduce that for such characters and such integers~$q,s$, the linking form~$B_3^\chi(q,s) \oplus B_4(q,s)$ is not metabolic.
  This is the desired contradiction, and Theorem~\ref{thm:LinIndep} is proved.
\end{proof}

\bibliography{BiblioAlgebraicKnots}

\begin{thebibliography}{10}

\bibitem{AbeTagami}
Tetsuya Abe and Keiji Tagami.
\newblock Fibered knots with the same 0-surgery and the slice-ribbon
  conjecture.
\newblock {\em Math. Res. Lett.}, 23(2):303--323, 2016.

\bibitem{Baker}
Kenneth~L. Baker.
\newblock A note on the concordance of fibered knots.
\newblock {\em J. Topol.}, 9(1):1--4, 2016.

\bibitem{BorodzikConwayPolitarczyk}
Maciej Borodzik, Anthony Conway, and Wojciech Politarczyk.
\newblock {Twisted Blanchfield pairings, twisted signatures and Casson-Gordon
  invariants}, 2018.

\bibitem{CassonGordon1}
Andrew~J. Casson and Cameron~McA. Gordon.
\newblock On slice knots in dimension three.
\newblock In {\em Algebraic and geometric topology \textup{(}{P}roc. {S}ympos.
  {P}ure {M}ath., {S}tanford {U}niv., {S}tanford, {C}alif., 1976\textup{)},
  {P}art 2}, Proc. Sympos. Pure Math., XXXII, pages 39--53. Amer. Math. Soc.,
  Providence, R.I., 1978.

\bibitem{CassonGordon2}
Andrew~J. Casson and Cameron~McA. Gordon.
\newblock Cobordism of classical knots.
\newblock In {\em \`A la recherche de la topologie perdue}, volume~62 of {\em
  Progr. Math.}, pages 181--199. Birkh\"auser Boston, Boston, MA, 1986.
\newblock With an appendix by P. M. Gilmer.

\bibitem{FriedlEta}
Stefan Friedl.
\newblock Eta invariants as sliceness obstructions and their relation to
  {C}asson-{G}ordon invariants.
\newblock {\em Algebr. Geom. Topol.}, 4:893--934, 2004.

\bibitem{FriedlVidussiSurvey}
Stefan Friedl and Stefano Vidussi.
\newblock A survey of twisted {A}lexander polynomials.
\newblock In {\em The mathematics of knots}, volume~1 of {\em Contrib. Math.
  Comput. Sci.}, pages 45--94. Springer, Heidelberg, 2011.

\bibitem{Hatcher}
Allen Hatcher.
\newblock {\em Algebraic topology}.
\newblock Cambridge University Press, Cambridge, 2002.

\bibitem{HeddenSomeRemarks}
Matthew Hedden.
\newblock Some remarks on cabling, contact structures, and complex curves.
\newblock In {\em Proceedings of {G}\"{o}kova {G}eometry-{T}opology
  {C}onference 2007}, pages 49--59. G\"{o}kova Geometry/Topology Conference
  (GGT), G\"{o}kova, 2008.

\bibitem{HeddenOnKnotFloer}
Matthew Hedden.
\newblock On knot {F}loer homology and cabling. {II}.
\newblock {\em Int. Math. Res. Not. IMRN}, (12):2248--2274, 2009.

\bibitem{HeddenNotions}
Matthew Hedden.
\newblock Notions of positivity and the {O}zsv\'{a}th-{S}zab\'{o} concordance
  invariant.
\newblock {\em J. Knot Theory Ramifications}, 19(5):617--629, 2010.

\bibitem{HeddenKirkLivingston}
Matthew Hedden, Paul Kirk, and Charles Livingston.
\newblock Non-slice linear combinations of algebraic knots.
\newblock {\em J. Eur. Math. Soc. \textup{(}JEMS\textup{)}}, 14(4):1181--1208,
  2012.

\bibitem{HeraldKirkLivingston}
Christopher Herald, Paul Kirk, and Charles Livingston.
\newblock Metabelian representations, twisted {A}lexander polynomials, knot
  slicing, and mutation.
\newblock {\em Math. Z.}, 265(4):925--949, 2010.

\bibitem{HomANoteOnCabling}
Jennifer Hom.
\newblock A note on cabling and {$L$}-space surgeries.
\newblock {\em Algebr. Geom. Topol.}, 11(1):219--223, 2011.

\bibitem{Kearton}
Cherry Kearton.
\newblock Cobordism of knots and {B}lanchfield duality.
\newblock {\em J. London Math. Soc. (2)}, 10(4):406--408, 1975.

\bibitem{KirkLivingston}
Paul Kirk and Charles Livingston.
\newblock Twisted {A}lexander invariants, {R}eidemeister torsion, and
  {C}asson-{G}ordon invariants.
\newblock {\em Topology}, 38(3):635--661, 1999.

\bibitem{Kitano}
Teruaki Kitano.
\newblock Twisted {A}lexander polynomial and {R}eidemeister torsion.
\newblock {\em Pacific J. Math.}, 174(2):431--442, 1996.

\bibitem{LickorishIntroduction}
William~B.R. Lickorish.
\newblock {\em An introduction to knot theory}, volume 175 of {\em Graduate
  Texts in Mathematics}.
\newblock Springer-Verlag, New York, 1997.

\bibitem{Litherland-signature}
Richard~A. Litherland.
\newblock Signatures of iterated torus knots.
\newblock In {\em Topology of low-dimensional manifolds \textup{(}{P}roc.
  {S}econd {S}ussex {C}onf., {C}helwood {G}ate, 1977\textup{)}}, volume 722 of
  {\em Lecture Notes in Math.}, pages 71--84. Springer, Berlin, 1979.

\bibitem{Litherland}
Richard~A. Litherland.
\newblock Cobordism of satellite knots.
\newblock In {\em Four-manifold theory \textup{(}{D}urham, {N}.{H}., 1982},
  volume~35 of {\em Contemp. Math.}, pages 327--362. Amer. Math. Soc.,
  Providence, RI, 1984.

\bibitem{LivingstonMelvinAlgebraicKnots}
Charles Livingston and Paul Melvin.
\newblock Algebraic knots are algebraically dependent.
\newblock {\em Proc. Amer. Math. Soc.}, 87(1):179--180, 1983.

\bibitem{LivingstonMelvin}
Charles Livingston and Paul Melvin.
\newblock Abelian invariants of satellite knots.
\newblock In {\em Geometry and topology \textup{(}{C}ollege {P}ark, {M}d.,
  1983/84\textup{)}}, volume 1167 of {\em Lecture Notes in Math.}, pages
  217--227. Springer, Berlin, 1985.

\bibitem{MillerPowell}
Allison~N. Miller and Mark Powell.
\newblock Symmetric chain complexes, twisted blanchfield pairings, and knot
  concordance.
\newblock {\em Algebr. Geom. Topol.}, 18(6):3425--3476, 2018.

\bibitem{Miyazaki}
Katura Miyazaki.
\newblock Nonsimple, ribbon fibered knots.
\newblock {\em Trans. Amer. Math. Soc.}, 341(1):1--44, 1994.

\bibitem{OSS}
Peter~S. Ozsv\'{a}th, Andr\'{a}s~I. Stipsicz, and Zolt\'{a}n Szab\'{o}.
\newblock Concordance homomorphisms from knot {F}loer homology.
\newblock {\em Adv. Math.}, 315:366--426, 2017.

\bibitem{Rudolph}
Lee Rudolph.
\newblock How independent are the knot-cobordism classes of links of plane
  curve singularities\textup{?}
\newblock {\em Notices Amer. Math. Soc.}, 23:410, 1976.

\bibitem{Tange}
Motoo Tange.
\newblock Upsilon invariants of {L}-space cable knots.
\newblock ar{X}iv:1703.08828, 2017.

\end{thebibliography}
\bibliographystyle{plain}

\end{document}